\documentclass[12pt]{article}

\usepackage{latexsym}
\usepackage{amsmath}
\usepackage{dsfont}
\usepackage{amsthm}
\usepackage{amsfonts}
\usepackage{amssymb}
\usepackage{psfrag}
\usepackage{graphicx}
\usepackage{verbatim} 
\usepackage{color}
\usepackage{booktabs}
\usepackage{tikz}
\usepackage{subcaption}
\usepackage{multirow, array} 
\usepackage{float} 

\usepackage[margin=1in]{geometry}
\allowdisplaybreaks

\renewenvironment{description}[1][0pt]
  {\list{}{\labelwidth=0pt \leftmargin=#1
   }}
  {\endlist}

\newtheorem{theorem}{Theorem}

\begin{document}

\title{\bf Some stochastic process techniques applied to deterministic models}
\author{
Eric Jos\'e \'Avila-Vales\\
Facultad de Matem\'aticas\\
Universidad Aut\'onoma de Yucat\'an\\
Anillo Perif\'erico Norte, Tablaje Catastral 13615\\ 
M\'erida, C.P. 97119, Yucat\'an, M\'exico\\
{\tt avila@correo.uady.mx}
\and 
José Villa-Morales\\
Departamento de Matemáticas y Física\\
Universidad Autónoma de Aguascalientes\\
Av. Universidad No. 940\\
Aguascalientes, Ags., México\\
{\tt jvilla@correo.uaa.mx }}
\date{\today}
\maketitle

\begin{abstract}

Stochastic mathematical models are essential tools for understanding and predicting complex phenomena. The purpose of this work is to study the exit times of a stochastic dynamical system—specifically, the mean exit time and the distribution of exit times of the stochastic process within a bounded domain. These quantities are obtained by solving elliptic and parabolic partial differential equations (PDEs), respectively. To support practical applications, we propose a numerical scheme implemented in FreeFEM, emphasizing its effectiveness in two- and three-dimensional cases due to the software's limitations in higher dimensions. The examples provided illustrate the theoretical results, which extend known one-dimensional solutions to higher-dimensional settings. This contribution bridges theoretical and computational approaches for analyzing stochastic processes in multidimensional domains, offering insights into their behavior and potential applications.

\end{abstract}

\maketitle

\section{Introduction}

Currently, many decision‐making processes rely on mathematical models, underscoring the importance of their rigorous study.  A well‐designed mathematical model aims to describe the phenomenon of interest as comprehensively as possible, incorporating only those parameters and variables that have a significant impact and omitting extraneous factors that would render the model impractical (see \cite{isham2005stochastic}).  

On the other hand, as is often the case in time‐dependent models, the system’s conditions may change from one day to the next.  Consequently, the predictions of a purely deterministic model can become unreliable or irrelevant (see \cite{pinsky2010introduction,morgan2008applied}).  A common—and sometimes effective—remedy is to adopt a model that accounts for such variations, including the uncertainty arising from imprecise parameter measurements (see \cite{goel2013stochastic}).  

Thus, if we assume that a deterministic model captures the “average" behavior of a phenomenon, then, in general terms, a stochastic model can be viewed as accounting for unpredictable perturbations or effects that are difficult to include explicitly.  Our first step, therefore, is to construct a stochastic counterpart.  There are multiple methods to achieve this (for instance, see \cite{mao2007stochastic}).  When the underlying dynamics are governed by a system of ordinary differential equations, a standard approach is to posit that, over small time intervals, the system undergoes random state transitions with prescribed probabilities (see \cite{allen2008construction,allen2007modeling}).  In Section \ref{SecMod}, we demonstrate—in a generic example—how to derive a stochastic dynamical system from a deterministic one (see Theorem \ref{ThExEDE}).  

Once we have the stochastic model, we use stochastic calculus techniques to study some functionals of interest. Namely, given a region of particular interest (D), where the process starts, we determine the average time the process takes to leave that region (Theorem \ref{TE}) and, similarly, we find the distribution of that exit time (Theorem \ref{TP}). These quantities depend on the solution of certain Partial Differential Equations (PDEs), the first of an Elliptic Equation (\ref{PDir}) and the second of a Parabolic Equation (\ref{PVI}), see \cite{karatzas1991brownian,friedman1975stochastic}.

The next step is to propose a solution scheme for these equations, see Section \ref{SecNum}. Since our purpose is for the theoretical results to be applied by a broad academic sector, we undertook the task of investigating the various software available to solve PDEs. Among them, some are paid and others are free access, the vast majority. We looked for, among other software characteristics, that it had certain previous versions to ensure the stability and quality of the results, that it was easy to code, and that it was used by a large community to ensure the existence of support forums. Among them, we chose FreeFEM, version 4.13, for being free software and meeting the aforementioned characteristics. Moreover, it is worth noting that some artificial intelligences know the FreeFEM language, which helps in coding. It is important to mention that, of all the software consulted, none of them solve PDE problems in dimensions greater than 4. There are two reasons for this: the first is that most applications focus on physics, and the second is that as dimensions increase, numerical methods quickly lose precision. Therefore, if a specific problem of interest arises in higher dimensions, a numerical method from the many that exist can be applied, see \cite{smith1985numerical}. The theoretical results considered in this work can be framed in any dimension; however, the limitation of the software to solve PDEs in dimensions no greater than 4 means that our exposition is limited to dimensions 2 and 3. The case of dimension 1 is well known, see \cite{oksendal2013stochastic}, which is why we have omitted it.

In summary, our contribution is as follows: given a stochastic dynamical system and a specified starting region, we employ partial differential equations and numerical‐analysis techniques to study the mean exit time and its distribution from that region.  To the best of our knowledge, these results are well established in one dimension but have not been derived—or at least not implemented—for higher‐dimensional settings.  To demonstrate the effectiveness of our approach, Section \ref{SecEx} presents illustrative examples from the literature that showcase the insights provided by our theoretical findings.

The work is organized as follows. In Section \ref{SecMod}, we introduce the stochastic model and show that it has a solution, $Y$. If $Y(0) \in D$, where $D$ is a bounded domain, in Section \ref{SecPDE}, we determine the expected mean exit time of $Y$ from the set $D$ and its distribution. In Section \ref{SecNum}, the numerical scheme is discussed, and the FreeFEM code is generally described, see \cite{MR3043640}. Finally, in Section \ref{SecEx}, several examples taken from the literature, are presented, and their solutions are discussed.

\section{Some Stochastic Models} \label{SecMod}

As noted in Section 1, the results we develop can be applied broadly to the solution of any SDE.  As a representative example, we employ the construction method described in \cite{allen2007modeling} to derive a stochastic model. Let us consider a phenomenon whose dynamic $Y(t)=(y_{1}(t),y_{2}(t),y_{3}(t))^{T}$ is governed by the following system of differential equations
\begin{eqnarray*} 
\frac{dy_{1}}{dt} & = & f_{1}(y_{1}) - c_{1,1}g_{1}(y_{1},y_{2}) - c_{1,2}g_{2}(y_{1},y_{3}), \\
\frac{dy_{2}}{dt} & = & f_{2}(y_{2}) - c_{2,1}g_{1}(y_{1},y_{2}) - c_{2,2}g_{3}(y_{2},y_{3}), \\ \frac{dy_{3}}{dt} & = & f_{3}(y_{3}) - c_{3,1}g_{3}(y_{2},y_{3}) - c_{3,2}g_{2}(y_{1},y_{3}), 
\end{eqnarray*} 
where $c_{i,j}\geq 0$ are real constants, and $f_{i}:\mathbb{R} \to \mathbb{R}$, $g_{i}:\mathbb{R}\times \mathbb{R} \to \mathbb{R}$ are non-negative continuous functions.

The components of the deterministic system can be described as follows:
\begin{description} 
\item[•] The function $f_{i}(y_{i})$ represents the growth rate of the variable $y_{i}$, independent of other variables.
\item[•] The function $g_{n}(y_{i},y_{j})$ represents the reduction in the growth rate of the variable $y_{m}$ due to the influence of the variables $y_{i}$ and $y_{j}$. This contribution is weighted by the constant $c_{i,j}$.
\end{description}

Suppose that over a small time interval $\Delta t$, the vector $\Delta Y$ can take only the following seven states:
\begin{gather*} 
(1,0,0)^{T}, \quad (0,1,0)^{T}, \quad (0,0,1)^{T}, \quad (-c_{1,1},-c_{2,1},0)^{T}, \\
(-c_{1,2},0,-c_{3,2})^{T}, \quad (0,-c_{2,2},-c_{3,1})^{T}, \quad (0,0,0)^{T}. 
\end{gather*} 
The probabilities associated with these states are shown in Table \ref{TaPr}.
\begin{table}[H] 
\centering 
\begin{tabular}{ l | l } Change & Probability \\
\hline \hline $\Delta Y^{1}= (1,0,0)^{T}$ & $p_{1}:= f_{1}(y_{1}) \Delta t$ \\
$\Delta Y^{2}=(0,1,0)^{T}$ & $p_{2}:= f_{2}(y_{2}) \Delta t$ \\
$\Delta Y^{3}=(0,0,1)^{T}$ & $p_{3}:= f_{3}(y_{3}) \Delta t$ \\
$\Delta Y^{4}=(-c_{1,1},-c_{2,1},0)^{T}$ & $p_{4}:= g_{1}(y_{1},y_{2}) \Delta t$ \\
$\Delta Y^{5}=(-c_{1,2},0,-c_{3,2})^{T}$ & $p_{5}:= g_{2}(y_{1},y_{3}) \Delta t$ \\
$\Delta Y^{6}=(0,-c_{2,2},-c_{3,1})^{T}$ & $p_{6}:= g_{3}(y_{2},y_{3}) \Delta t$ \\
$\Delta Y^{7}=(0,0,0)^{T}$ & $p_{7}:=1- \sum_{i=1}^{6} p_{i}$ 
\end{tabular} 
\caption{Possible changes with their corresponding probabilities.} \label{TaPr} 
\end{table}
From this, the drift vector is given by
\begin{equation*} 
B = \left( \begin{array}{c} b_{1} \\ b_{2} \\ b_{3} \end{array} \right) 
:= \frac{E[\Delta Y]}{\Delta t} = \frac{1}{\Delta t} \sum_{i=1}^{7} p_{i} \Delta Y^{i} 
= \left( \begin{array}{c} f_{1} - c_{1,1}g_{1} - c_{1,2}g_{2} \\ 
f_{2} - c_{2,1}g_{1} - c_{2,2}g_{3} \\
f_{3} - c_{3,1}g_{3} - c_{3,2}g_{2}
\end{array} \right), 
\end{equation*} 
and the diffusion matrix is
\begin{eqnarray} 
A =(a_{i,j}) & := & \frac{E[\Delta Y (\Delta Y)^{T}]}{\Delta t} = \frac{1}{\Delta t} \sum_{i=1}^{7} p_{i} \Delta Y^{i} (\Delta Y^{i})^{T} \nonumber\\
&=& \left( \begin{array}{ccc} f_{1} + c_{1,1}^{2}g_{1} + c_{1,2}^{2}g_{2} & c_{1,1}c_{2,1}g_{1} & c_{1,2}c_{3,2}g_{2}  \\ 
c_{2,1}c_{1,1}g_{1} & f_{2} + c_{2,1}^{2}g_{1} + c_{2,2}^{2}g_{3} & c_{2,2}c_{3,1}g_{3}  \\
c_{3,2}c_{1,2}g_{2} & c_{3,1}c_{2,2}g_{3} & f_{3} + c_{3,2}^{2}g_{2} + c_{3,1}^{2}g_{3}
\end{array} \right). \label{matA}
\end{eqnarray}
Note that $A$ is positive definite, so it has a unique square root $\sigma$, such that $A = \sigma \sigma^{T}$.

Let $p(t,x_{1},x_{2},x_{3})$ be the probability that the process $X$ is in state $(x_{1},x_{2},x_{3})$ at time $t$. Using the transition probabilities from Table \ref{TaPr} and letting $\Delta t$ approach $0$, it can be shown that the probabilities $p(t,x_{1},x_{2},x_{3})$ satisfy the Fokker-Planck equation (see page 139 in \cite{allen2007modeling}): 
\begin{eqnarray*} 
\frac{\partial p(t,x_{1},x_{2},x_{3})}{\partial t} 
& = & \frac{1}{2} \sum_{i=1}^{3} \sum_{j=1}^{3} \frac{\partial^{2}}{\partial x_{i} \partial x_{j}} \left[ \sum_{k=1}^{3} \sigma_{i,k}(x_{1},x_{2},x_{3})\sigma_{j,k}(x_{1},x_{2},x_{3})p(t,x_{1},x_{2},x_{3}) \right] \\
&& - \sum_{i=1}^{3} \frac{\partial}{\partial x_{i}}[ b_{i}(x_{1},x_{2},x_{3})p(t,x_{1},x_{2},x_{3})]. 
\end{eqnarray*}

On the other hand, the probabilities $p(t,x_{1},x_{2},x_{3})$ correspond to the solution of the stochastic differential equation (see Theorem 10.9.10 in \cite{kuo2006stochastic}) \begin{equation} \label{sisEDEw} 
\left\lbrace 
\begin{tabular}{ l l l l } 
$dY(t)$ & $=$ & $B(Y(t))dt+ \sigma (Y(t))dW(t),$ & $t>0$, \\ 
$Y(0)$ & $=$ & $y^{0}$, &
\end{tabular} \right. 
\end{equation} 
where $W=(W_{1},W_{2},W_{3})^{T}$ is a Brownian motion in $\mathbb{R}^{3}$.

Next, we will show that the system (\ref{sisEDEw}) indeed has a unique solution.

\begin{theorem} \label{ThExEDE} 
Suppose that the functions $f_{i}$ and $g_{i}$ are of class $C^{2}(\mathbb{R})$ and $C^{2}(\mathbb{R}^{2})$, respectively, for $i=1,2,3$. Then the system (\ref{sisEDEw}), with initial condition $Y(0)=y^{0}$, has a unique solution in the strong sense. That is, there exists a probability space $(\Omega, \mathcal{F},P)$, a Brownian motion $W$, and a stochastic process $Y={Y(t): 0\leq t < \infty }$ such that: 
\begin{description} 
\item[(i)] $Y$ is adapted to the filtration generated by $W$. 
\item[(ii)] $P[\int_{0}^{t}{|b_{i}(Y(s))|+\sigma_{i,j}^{2}(Y(s))}ds< \infty]=1$ holds for every $i,j=1,2,3$ and $0<t<\infty$. 
\item[(iii)] Moreover, if $Y^{0}=(y_{1}^{0},y_{2}^{0},y_{3}^{0})^{T}$ and $Y(t)=(y_{1}(t),y_{2}(t),y_{3}(t))^{T}$, then 
\begin{equation}\label{sisEDEI} 
Y(t) = Y^{0} + \int_{0}^{t}B(Y(s)) ds + \int_{0}^{t} \sigma (Y(s)) dW(s), 
\end{equation} 
holds almost surely. Recall that the stochastic integral is in the It\^o sense. 
\end{description} 
\end{theorem}

\begin{proof} 
The first thing to observe is that the smoothness conditions of the components $a_{i,j}$ of $A$ are inherited by the components $\sigma_{i,j}$ of $\sigma$. This can be verified in Theorem 1.2 of Chapter 6 of \cite{friedman1975stochastic}. Thus, in our case, the hypotheses on the functions $f_{i}$ and $g_{i}$ imply that the functions $b_{i}$ and $\sigma_{i,j}$ are smooth functions. Therefore, the uniqueness of the solution follows immediately from the fact that every smooth function is locally Lipschitz, see for example Theorem 2.5 of Chapter 5 of \cite{karatzas1991brownian}. Furthermore, it follows that the system (\ref{sisEDEI}) has a strong solution, which may explode in finite time. For details, see for example Section 4.3 of Chapter 4 of \cite{mckean2024stochastic}. 
\end{proof}

It is known, see Theorem 2.9 of Chapter 5 of \cite{karatzas1991brownian}, that when the components of the drift vector $B$ and the components of the dispersion matrix $\sigma$ have linear growth, the system (\ref{sisEDEI}) has a global solution, that is, $Y(t)\in \mathbb{R}^{3}$ for every $t>0$. However, in our case, where the components of the vector and the matrix are smooth functions, explosion in finite time may occur. That is, there exists a stopping time $\mathfrak{e}$ such that $|Y(\mathfrak{e-})|=\infty$ on the event ${\mathfrak{e}<\infty}$, see \cite{gikhman2007stochastic}.

Except in the case of dimension $2$, it is generally not easy to explicitly determine the positive root of the matrix $A$. Fortunately, for the applications we have in mind, it is not necessary to know the root explicitly. It is sufficient to know that the positive root $\sigma$ exists and inherits the smoothness properties of the components of $A$ to such an extent that (\ref{sisEDEI}) has a unique strong solution. It should be noted that other approaches require increasing variability to consider a stochastic model. That is, they require introducing a Brownian motion in dimensions greater than or equal to $3$, see for example \cite{allen2008construction}.

\section{Connection with Partial Differential Equations} \label{SecPDE}

In the following, we will see how PDEs relate to the stochastic model.

Let $D \subset \mathbb{R}^{3}$ be a bounded domain (a set that is open and connected). Consider the differential operator 
\begin{equation} 
Lu := \frac{1}{2} \sum_{i=1}^{3} \sum_{j=1}^{3} a_{i,j} \ \frac{\partial^{2} u}{\partial x_{i}\partial x_{j}} + \sum_{i=1}^{3} b_{i} \ \frac{\partial u}{\partial x_{i}}, \quad u \in C^{2}(D). \label{GIE} 
\end{equation} 
We say that a point $x \in \partial D$ satisfies the exterior sphere condition if there exists an open ball centered at $x$, $B(x)$, such that $\overline{B(x)} \cap D = \emptyset$ and $\overline{B(x)} \cap \overline{D} = {x}$. With these preliminaries, we can state the following result, commonly known as the Dirichlet Problem.

\begin{theorem} \label{THExE} 
Assume that the functions $f_{i}$, $g_{i}$ satisfy the hypotheses of Theorem \ref{ThExEDE}. Additionally, suppose that $\inf_{x \in \overline{D}} a_{i,j}(x) > 0$, $i,j \in {1,2,3}$. If the points of $\partial D$ satisfy the exterior sphere condition, then there exists a unique function $u$ of class $C(\overline{D}) \cap C^{2}(D)$ such that 
\begin{equation}\label{PDir} 
\left\lbrace \begin{tabular}{ l l } $Lu(x) = -1$, & $x \in D$, \\ 
$u(x) = 0$, & $x \in \partial D$.
\end{tabular} \right. 
\end{equation} 
\end{theorem}

\begin{proof} 
Let $\mu := \min_{i,j \in {1,2,3}} \inf_{x \in \overline{D}} a_{i,j}(x)$, then 
\begin{equation} 
\sum_{i=1}^{3} \sum_{j=1}^{3} a_{i,j} \ \xi_{i} \xi_{j} \geq \mu |\xi|^{2}, \label{posdef} \end{equation} 
for every $\xi = (\xi_{1}, \xi_{2}, \xi_{3})^{T} \in \mathbb{R}^{3}$. This implies that the operator $L$ is uniformly elliptic in $D$. Moreover, since the functions $f_{i}$, $g_{i}$ are smooth and $D$ is bounded, the coefficients $b_{i}$ and $a_{i,j}$ are Lipschitz continuous. Furthermore, since the points of $\partial D$ satisfy the exterior sphere condition, the required conditions for the Dirichlet problem are met. Therefore, there exists a unique solution in the desired function space (see page 101 in \cite{gilbarg1977elliptic} or page 87 in \cite{friedman2008partial}). 
\end{proof}

Consider the stochastic process $Y$ from Theorem \ref{ThExEDE}. We define the first exit time of the process $Y$ from the bounded domain $D$ as the stopping time $$\tau_{D} = \inf {t \geq 0: Y(t) \notin D}.$$ Since $\inf_{x \in \overline{D}} a_{i,i}(x) > 0$, it follows (see Lemma 7.4, Chapter 5 of \cite{karatzas1991brownian}) that 
\begin{equation} 
E^{x}[\tau_{D}] < \infty, \quad \text{for all} \quad x \in D, \label{MedH} 
\end{equation} 
where the expectation $E^{x}$ is taken with respect to the probability measure $P^{x}[\cdot] := P[\cdot | Y(0) = x]$, from Theorem \ref{ThExEDE}. Furthermore, this expected value is determined by the solution of Theorem \ref{THExE}.

\begin{theorem} \label{TE} 
Let $u$ be the solution to the Dirichlet problem (\ref{PDir}), assuming the previous hypotheses for the existence of such a solution. Then 
\begin{equation} 
u(x) = E^{x}[\tau_{D}], \quad \text{for all} \quad x \in \overline{D}. \label{VEH} 
\end{equation} 
\end{theorem}

\begin{proof} 
From (\ref{MedH}), we know that $\tau_{D} < \infty$ almost surely. Let $\tau_{n} := \inf {t : \text{dist}(Y(t), \partial D) < 1/n}$, where $\text{dist}(\cdot, \cdot)$ is the distance between sets. Let $u$ be the unique solution to (\ref{PDir}). By Itô’s formula, 
\begin{eqnarray*} 
u(Y(t \wedge \tau_{n})) &=& u(Y(0)) + \text{martingale} + \int_{0}^{t \wedge \tau_{n}} (Lu)(Y(s)) , ds \\
 &=& u(x) + \text{martingale} - (t \wedge \tau_{n}). 
\end{eqnarray*} 
Taking the expected value, we obtain 
$$u(x) = E^{x}[u(Y(t \wedge \tau_{n}))] + E^{x}[(t \wedge \tau_{n})].$$ 
Letting $\tau_{n} \uparrow \tau_{D}$ and $t \uparrow \infty$, we arrive at (\ref{VEH}) using the Dominated Convergence Theorem and the boundary condition of $u$. 
\end{proof}

Thus, we see that the expected exit time of the process $Y$ from the domain $D$ is determined by an elliptic Partial Differential Equation. We will now show that the distribution of this exit time depends on the solution to a parabolic Partial Differential Equation.

We know that $\tau_{D} < \infty$ almost surely, meaning that if the process $Y$ starts inside $D$, almost surely all trajectories will exit the domain in finite time. We will now compute its distribution, but first, it is necessary to introduce the partial differential operator \begin{equation*} 
\overline{L} v := L v - \frac{\partial v}{\partial t}, \quad v \in C^{2}(D \times (0, \infty)), \end{equation*} 
where $D \subset \mathbb{R}^{3}$ is as before, i.e., a bounded domain. The following result is known as the Initial Value Problem.

\begin{theorem} \label{THExP} 
Under the hypotheses of Theorem \ref{THExE}, there exists a unique function $v$ of class $C(\overline{D} \times (0, \infty)) \cap C^{2}(D \times (0, \infty))$ such that 
\begin{equation}\label{PVI} 
\left\lbrace 
\begin{tabular}{ l l } 
$\overline{L} v(x,t) = 0$, & $x \in D, \ t > 0$, \\ 
$v(x,t) = 0$, & $x \in \partial D, \ t \geq 0$, \\ 
$v(x,t) = 1$, & $x \in D, \ t = 0$. 
\end{tabular} \right. 
\end{equation} 
\end{theorem}

\begin{proof} 
From (\ref{posdef}), it follows that the operator $\overline{L}$ is uniformly parabolic in $D \times (0, \infty)$. Furthermore, the coefficients $a_{i,i}$ and $b_{i}$ are uniformly H\"older continuous on $\overline{D} \times [0, \infty)$. Since each point of $\partial D$ satisfies the exterior sphere condition, each point of $\partial D \times (0, \infty)$ has a barrier. Thus, there exists a unique real-valued function $v$ that satisfies (\ref{PVI}); see Theorem 7 in Section 3, Chapter 3 of \cite{friedman2008partial} or Theorem 3.6 in Section 3, Chapter 6 of \cite{friedman1975stochastic}. 
\end{proof}

This result allows us to determine the distribution of $\tau_{D}$. 

\begin{theorem} \label{TP} 
Assuming the hypotheses of Theorem \ref{THExE}, let $v$ be the solution to the Initial Value Problem (\ref{PVI}), then 
\begin{equation} 
v(x,t)=P^{x}[\tau_{D} > t], \quad \mbox{for all} \quad x\in \overline{D} \ \mbox{and} \ t>0. \label{DistEx} 
\end{equation} 
\end{theorem}

\begin{proof} 
Fix $t>0$ and let $\tau_{n}$ be the stopping time defined in the proof of Theorem \ref{THExE}. Let $v$ be the unique solution to (\ref{PVI}). By Itô's formula on $\mathbb{R}^{d}\times [0,t)$, \begin{eqnarray*} 
v(Y(s\wedge \tau_{n}),t-(s\wedge \tau_{n}))&=& v(Y(0),t) + \mbox{martingale} +\int_{0}^{t\wedge \tau_{n}} (Lv)(Y(r),t-r)dr \\ 
&&+ \int_{0}^{t\wedge \tau_{n}}\left( -\frac{\partial v}{\partial r}\right) (Y(r),t-r)dr. \end{eqnarray*} 
Taking expectations and letting $\tau_{n}\uparrow \tau_{D}$, the Dominated Convergence Theorem implies $$v(Y(0),t)=E[v(Y(s\wedge \tau_{D}),t-(s\wedge \tau_{D}))].$$ Letting $s\uparrow t$, it follows that 
\begin{eqnarray*} 
v(x,t)&=&E^{x}[v(Y(t\wedge \tau_{D}),t-(t\wedge \tau_{D}))]\\ 
&=&E^{x}[v(Y(t),0);\tau_{D} > t]+ E^{x}[v(Y(\tau_{D}),0);\tau_{D} = t] + E^{x}[v(Y(\tau_{D}),t-\tau_{D});\tau_{D} < t]. 
\end{eqnarray*} 
Thus, the equality (\ref{DistEx}) follows from the boundary conditions given in (\ref{PVI}). 
\end{proof}

\section{Numerical Aspect of Solving the PDEs (\ref{PDir}) and (\ref{PVI})} \label{SecNum}

To use FreeFEM, we need the variational form of equation (\ref{PDir}) (see \cite{evans2022partial}). For this, we take a test function $w$ and multiply both sides of equation (\ref{PDir}). Then, we integrate over the entire domain 
$$ \frac{1}{2} \sum_{i=1}^{3} \sum_{j=1}^{3} \int_{D} w\, a_{i,j} \frac{\partial^{2} u}{\partial x_{i} \partial x_{j}} dx+ \sum_{i=1}^{3} \int_{D} w\, b_{i}\frac{\partial u}{\partial x_{i}} dx = -\int_{D} w dx.$$ 
Using the integration by parts formula (see page 714 in \cite{evans2022partial}) we get \begin{eqnarray*} 
\int_{D} w\, a_{i,j} \frac{\partial^{2} u}{\partial x_{i} \partial x_{j}} dx
&=&-\int_{D} a_{i,j}\frac{\partial u}{\partial x_{i}} \, \frac{\partial w}{\partial x_{j}} dx- \int_{D} w \frac{\partial u}{\partial x_{i}} \, \frac{\partial a_{i,j}}{\partial x_{j}} dx\\ 
&& + \int_{\partial D} w\, a_{i,j} \frac{\partial u}{\partial x_{i}} \nu^{j} dS, 
\end{eqnarray*} 
where $\nu=(\nu^{1},...,\nu^{n})^{T}$ is the unit outward pointing vector and $dS$ is the surface measure. Due to the Dirichlet boundary conditions, we choose $w$ such that $w=0$ on $\partial D$. Consequently, we arrive at 
$$\int_{D}G(u,w)dx=\int_{D}F(w)dx,$$ 
where $G$ is the bilinear form 
\begin{equation} 
G(u,w):= \frac{1}{2} \sum_{i=1}^{3} \sum_{j=1}^{3} a_{i,j}\frac{\partial u}{\partial x_{i}} \, \frac{\partial w}{\partial x_{j}} + \frac{1}{2} \sum_{i=1}^{3}\sum_{j=1}^{3} w \frac{\partial u}{\partial x_{i}} \, \frac{\partial a_{i,j}}{\partial x_{j}} - \sum_{i=1}^{3} b_{i}\, w\frac{\partial u}{\partial x_{i}} \label{BFE} 
\end{equation} and $F(w)$ is the linear one 
$$ F(v):= w.$$
On the other hand, in what follows we will assume that the domain $D$ is of the form 
$$D=(a_{1},a_{2})\times (b_{1},b_{2}) \quad \text{or} \quad D=(a_{1},a_{2})\times (b_{1},b_{2}) \times (c_{1},c_{2}).
$$ Consequently, the set $\partial D$ satisfies the exterior sphere condition. It should be noted that FreeFEM supports much more general domains; we have chosen this one solely to demonstrate our results.

The algorithm in FreeFEM to solve the PDE (\ref{PDir}) consists of three main parts (see \cite{MR3043640}): 
\begin{description} 
\item[I.] Definition of the Domain $D$. 
\item[II.] Definition of the functions $b_{i}, a_{i,i}$ and $\partial a_{i,i}/\partial x_{i}$. \item[III.] Definition of the problem \\ 
 $\mathtt{int3d(Th)}\left(G(u,w)\right)$; \ bilinear form \\ 
 - $\mathtt{int3d(Th)}$($F(w)$); \ linear form \\ 
 + $\mathtt{on}$(1, 2, 3, 4, 5, 6, $u=g$); \ boundary condition. 
\end{description} 
In FreeFEM code, the integral of the bilinear form and the linear form must always be written separately. For greater utility of our results, we have shared the code of one of the examples we address in the next section in the Appendix.

Now let us address the case of the parabolic PDE (\ref{PVI}) (see Section 5.5 of \cite{MR3043640}). Take $T>0$ and $\eta\in (0,T)$. For $\eta$ close to $0$ we have that 
$$\frac{\partial v(x,t)}{\partial t} \approx \frac{v(x,t)-v(x,t-\eta)}{\eta}.$$ 
Therefore, if we set $u^{m}(x):=v(t,m\eta)$, $m=0,1,...,[T/\eta]$, where $[T/\eta]$ is the integer part of $T/\eta$, we get 
$$\frac{\partial v(x,m \eta)}{\partial t} \approx \frac{u^{m+1}(x)-u^{m}(x)}{\eta}.$$ 
Hence, the time discretization of equation (\ref{PVI}) is as follows, 
\begin{equation}\label{disPVI} 
\left\lbrace \begin{tabular}{ r c l l } 
$\frac{u^{m+1}(x)-u^{m}(x)}{\eta}$ &=& $L u^{m+1}$, & $x \in D$, \\ 
$u^{m+1}(x)$ &=& $0$, & $x \in \partial D$, \\ 
$u^{0}(x)$ &=& $1$, & $x \in D$. 
\end{tabular} \right.
\end{equation} 
We proceed as before, multiplying (\ref{disPVI}) by a test function $w$, such that $w=0$ on $\partial D$, and integrating over the entire domain to obtain 
$$\int_{D}G(u^{m+1},w)dx=\int_{D}F(w)dx,$$ 
where 
$$G(u^{m+1},w):= \frac{\eta}{2}\sum_{i=1}^{3} \sum_{j=1}^{3} a_{i,j} \frac{\partial u^{m+1}}{\partial x_{i}} \frac{\partial w}{\partial x_{j}} + \frac{\eta}{2}\sum_{i=1}^{3} \sum_{j=1}^{3} w \frac{\partial u^{m+1}}{\partial x_{i}} \frac{\partial a_{i,j}}{\partial x_{j}} - \eta \sum_{i=1}^{3} b_{i} w \frac{\partial u^{m+1}}{\partial x_{i}}+ u^{m+1} w $$ 
and 
$$F(w):= u^{m} w.$$

In this way, at each step we solve an Elliptic Equation, thus obtaining an approximate solution to the Parabolic Equation. Therefore, in addition to the previous steps of the FreeFEM scheme, we must add a fourth step. Namely, 
\begin{description} 
\item[IV.] Iterative loop. Solve the Elliptic PDE, applying steps ${\bf I}$-${\bf III}$, for $m=0,1,...,[T/\eta]$. 
\end{description} 
In the next section, we will present some examples and detail the FreeFEM code for only one of them, as the others are handled similarly.

\section{Examples} \label{SecEx}

To illustrate the application of our theoretical and numerical results, we consider four deterministic models and introduce a stochastic component into each.  Two of these models are two‐dimensional, and the other two are three‐dimensional.  These examples are drawn from the literature; our aim is not to evaluate their formulation, but to analyze the exit‐time functionals of the corresponding stochastic models.  Accordingly, for each reference, we specify the model under study and clearly define the meaning of every parameter.  This emphasizes how the results developed in this work can be applied in practice.  

\begin{description}

\item[A rumor model:] The rumor model has gained increasing importance, particularly due to its various real-world implications. For instance, it may represent the preference for a political party (Democrat or Republican) or the potential devaluation of a certain currency. Moreover, the integration of social networks into society has heightened the interest in systematically studying this phenomenon. In our case, we assume that the rumor is modeled by the following system of differential equations (see \cite{zhang2024stochastic}):
\begin{eqnarray*}
\frac{dS}{dt}&=& \Lambda+ \alpha I^{2}-\beta SI - \mu S ,\\
\frac{dI}{dt}&=&\beta SI- (\mu + \eta)I - \alpha I^{2}.
\end{eqnarray*}
We assume that the parameters $\Lambda$, $\alpha$, $\beta$, $\mu$, and $\eta$ are non-negative real numbers, whose meanings, which are not essential for our purposes, can be consulted in \cite{zhao2016dynamic}.

Let $Y=(S,I)^{T}$, where $S$ represents the percentage of individuals susceptible to the rumor and $I$ represents the percentage of individuals influenced by the rumor. As in Section \ref{SecMod}, we have that over a small time interval $\Delta t$, the vector $\Delta Y$ can take only the following five states:
\begin{gather*}
(-1,0)^{T},\ (0,-1)^{T}, \ (1,0)^{T}, \ (-1,1)^{T}, \ (0,0)^{T}.
\end{gather*}
Table \ref{TaRumor} provides the probabilities for each of these states.
\begin{table}[H]
\centering
\begin{tabular}{ l | l }
Change & Probability \\
 \hline \hline
 $\Delta Y^{1}= (-1,0)^{T}$ &  $p_{1}:= f_{1}(S) \Delta t$\\
 $\Delta Y^{2}=(0,-1)^{T}$ &  $p_{2}:= f_{2}(I) \Delta t$\\
 $\Delta Y^{3}=(1,0)^{T}$ &  $p_{3}:= g_{1}(I) \Delta t$\\
 $\Delta Y^{4}=(-1,1)^{T}$ &  $p_{4}:= g_{2}(S,I) \Delta t$\\
 $\Delta Y^{5}=(0,0)^{T}$ &  $p_{5}:= 1- \sum_{i=1}^{4} p_{i}$     
\end{tabular}
\caption{Probabilities for changes in the rumor model.} \label{TaRumor} 
\end{table}
In this case, the functions $f_{i}$ and $g_{i}$ are given by:
\begin{gather*} 
f_{1}(S)=\mu S, \quad f_{2}(I)= (\mu + \eta)I + \alpha I^{2}, \quad g_{1}(I)= \Lambda + \alpha I^{2},  \quad g_{2}(S,I)= \beta SI.
\end{gather*}
With this information, the stochastic system of interest is:
\begin{equation*}\label{EDESIS}
\left\lbrace 
\begin{tabular}{ l l  l}
 $dS(t)$ & $=$ &  $(\Lambda+ \alpha I^{2}-\beta SI - \mu S) dt + \sigma_{1,1}(S,I)dW_{1}(t)+\sigma_{1,2}(S,I)dW_{2}(t)$, \\ 
 $dI(t)$ & $=$ &  $(\beta SI- (\mu + \eta)I - \alpha I^{2}) dt + \sigma_{2,1}(S,I)dW_{1}(t)+\sigma_{2,2}(S,I)dW_{2}(t)$,
\end{tabular} \right.
\end{equation*}
where $(\sigma_{i,j})$ is the square root of:
\begin{eqnarray*}
(a_{i,j})=\left(
\begin{tabular}{ c c }
 $\mu S + \Lambda + \alpha I^{2} + \beta SI$ & $-\beta SI$ \\ 
 $-\beta SI$ &  $ (\mu + \eta)I + \alpha I^{2}+\beta SI$     
\end{tabular} \right).
\end{eqnarray*}

Now, we can compute the elliptic differential operator $L$ and the bilinear form $G$ (see \eqref{GIE} and \eqref{BFE}, respectively) and proceed with the numerical part.

In this case, we take a numerical example from \cite{zhao2016dynamic}. Specifically, the parameter values considered are:
$$\Lambda= 0.5, \quad \mu=0.3, \quad \eta = 0.2, \quad \alpha=0.1, \quad \beta=0.4.$$ 
Let us assume $S(0) = 0.8$, $I(0) = 0.2$, and $D=(0.7,0.9)\times (0.1,0.3)$. We use FreeFEM to solve equations \eqref{PDir}, \eqref{PVI}, and obtain the following results:
\begin{figure}[]
    \begin{subfigure}[bt]{0.55\textwidth}
    \centering
       \includegraphics[scale=0.16]{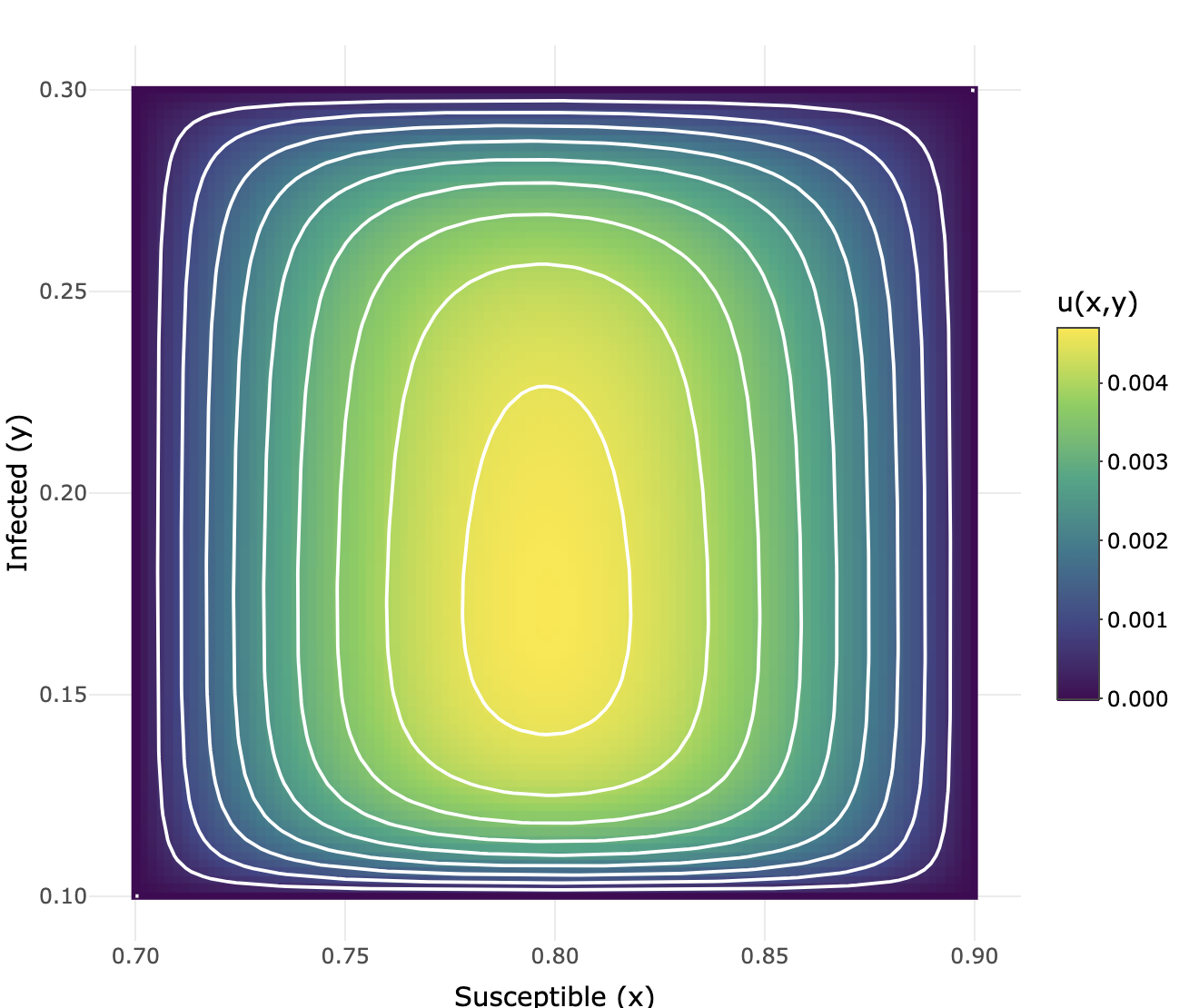}
       \caption{Level curves of $(x,y) \to E^{(x,y)}[\tau_{D}]$.}  
    \end{subfigure}    
    \begin{subfigure}[bt]{0.5\textwidth}
    \hspace{-0.3cm}
       \includegraphics[scale=0.16]{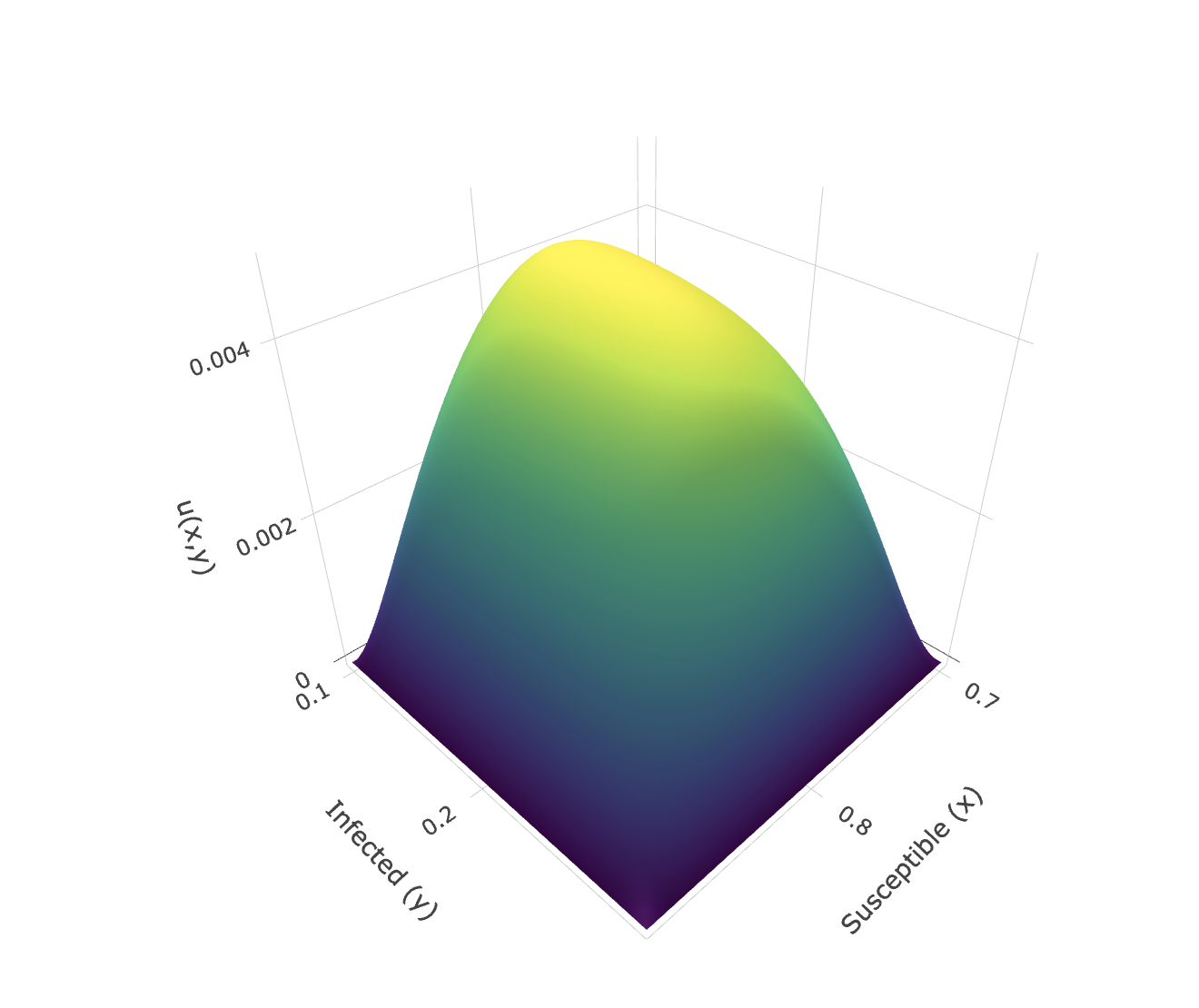}
       \caption{Graph of $(x,y) \to E^{(x,y)}[\tau_{D}]$.}  
    \end{subfigure}

     \begin{subfigure}[bt]{.55\textwidth}
    \hspace{-0.5cm}
       \includegraphics[scale=0.18]{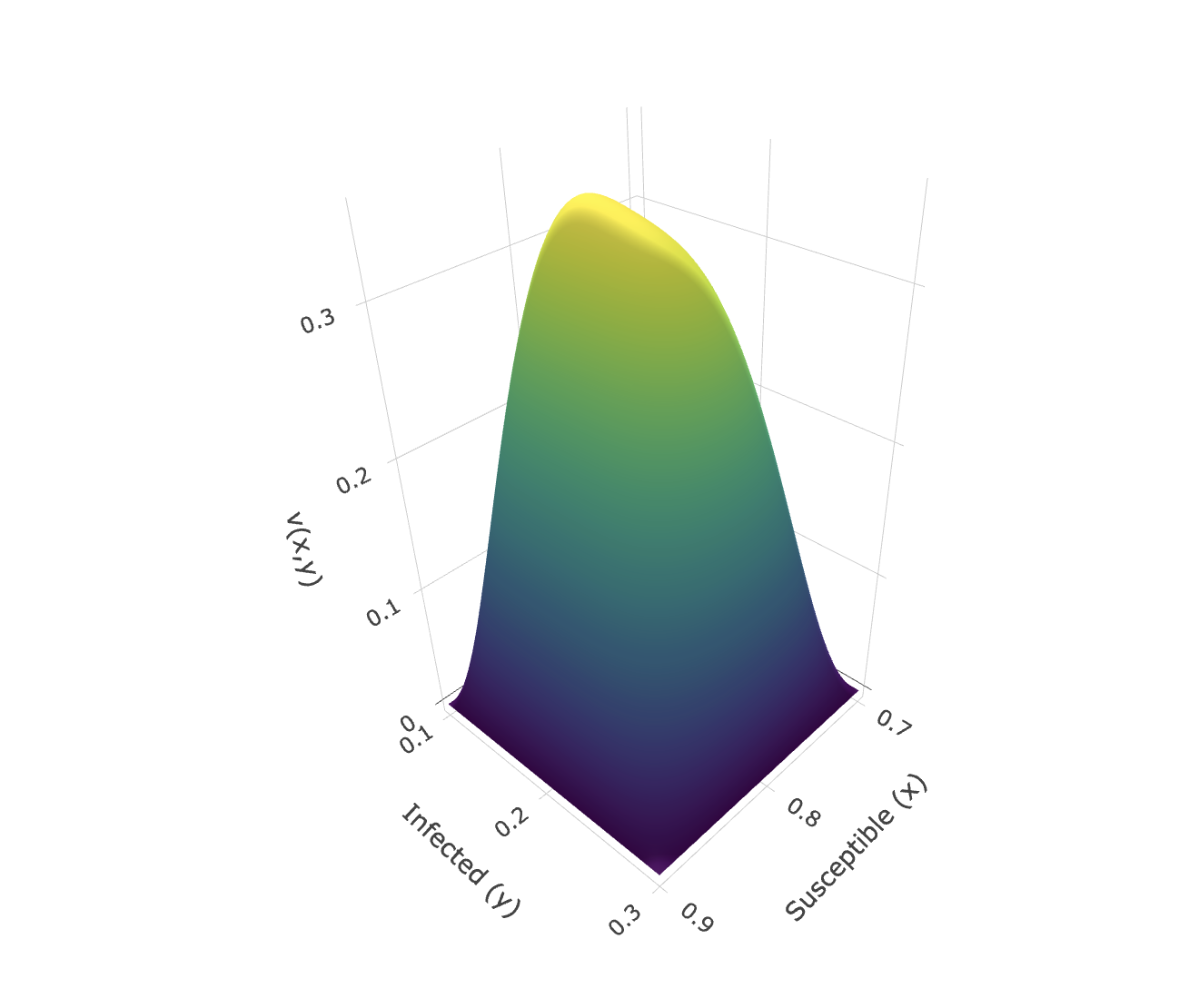}
       \caption{Graph of $(x,y) \to P^{(x,y)}[\tau_{D}>0.005]$.}  
    \end{subfigure} 
    \begin{subfigure}[bt]{0.4\textwidth}
    \hspace{-1.5cm}
    \vspace{0.6cm}
       \includegraphics[scale=0.16]{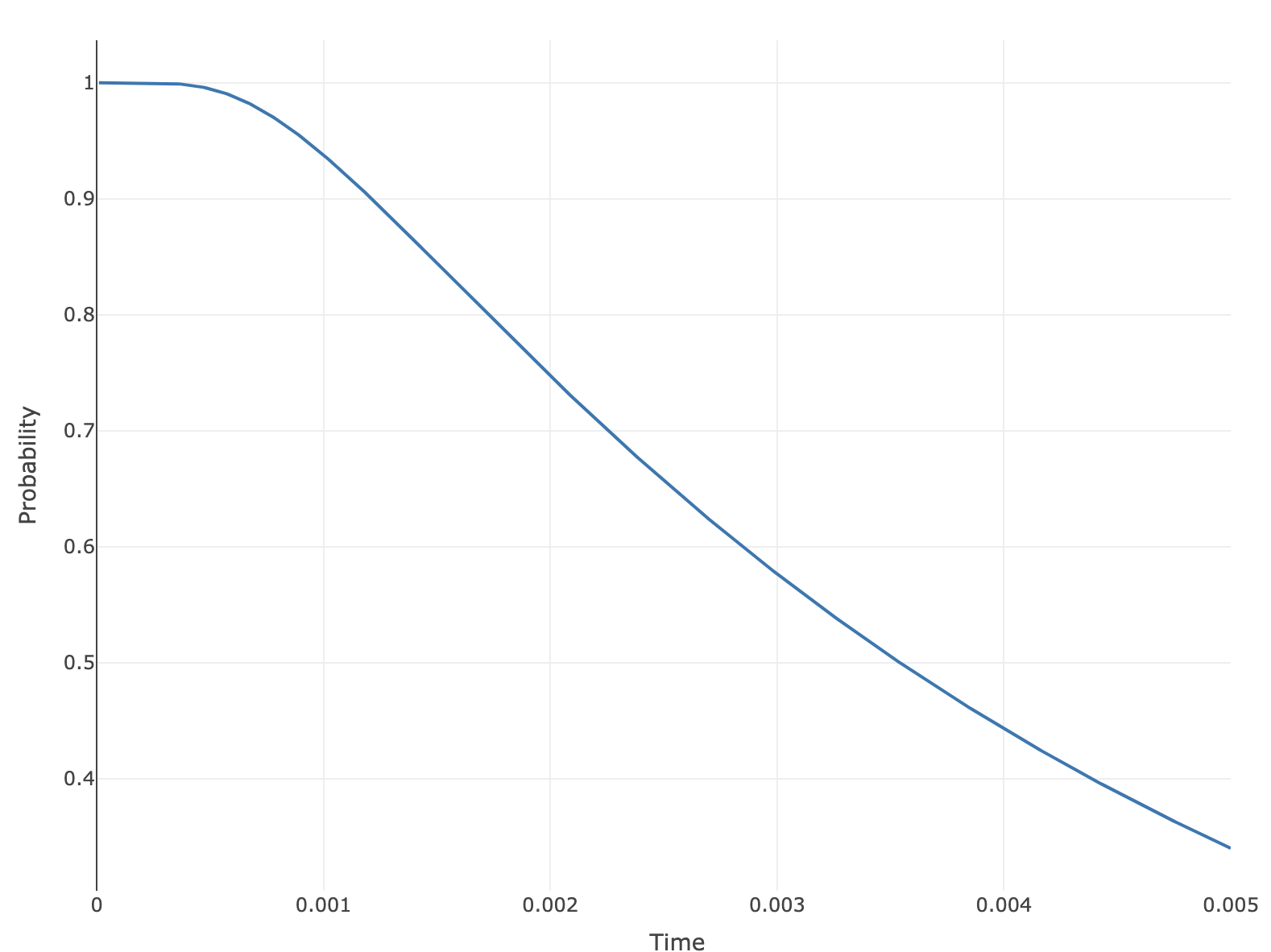}
       \caption{Graph of $t \to P^{(0.8,0.2)}[\tau_{D}>t]$.}  
    \end{subfigure}
   \caption{The expectation and probability distribution of the rumor exit time for $D=(0.7,0.9)\times (0.1,0.3)$.}  \label{FigRumor}  
\end{figure}

Figure \ref{FigRumor}(a) shows the contour lines of the function $u(x,y)$. From this, we observe that the first contour level is broad, while the subsequent levels shrink rapidly, indicating an accelerated evolution of the rumor. Specifically, the initial point $(0.8,0.2)$ lies on the first level, whereas the point $(0.75,0.25)$ lies on the second level. A small $5\%$ change in the rumor leads to a reduction of four levels. More precisely, from Figure \ref{FigRumor}(b), we find that $u(0.8,0.2) = 4.636 \times 10^{-3}$ and $u(0.75,0.25) = 3.235 \times 10^{-3}$, resulting in a $69\%$ decrease in the average exit time. This means that a slight increase in the rumor causes it to grow rapidly. Furthermore, Figure \ref{FigRumor}(b) also shows that the initial point is the maximum of $u$, as expected due to the symmetry of the domain $D$ and the condition $S + I = 1$.

On the other hand, Figure \ref{FigRumor}(c) reveals that $v(0.8,0.2) = 0.34$, which indicates that the probability of the rumor not exceeding $10\%$ of its initial value within more than $0.005$ time units is $34\%$. Similarly, $v(0.75,0.25) = 0.211$ suggests that the probability of the rumor remaining within the range $(0.1,0.3)$ for more than $0.005$ time units is approximately $21\%$. As before, this underscores the importance of containing the rumor at early stages. A $5\%$ increase in the rumor reduced the probability of staying in the range $(0.1,0.3)$ by $13\%$.

From Figure \ref{FigRumor}, we know that the maximum average value of $\tau_{D}$ is $4.636 \times 10^{-3}$ units, providing insight into the distribution domain of $P^{(0.8,0.2)}[\tau_{D}>t]$. Figure \ref{FigRumor}(d) shows that after $0.005$, the probabilities become very small. Additionally, we noticed that at the beginning of the rumor, its growth is extremely slow; for instance, in the time interval $(0,5.2 \times 10^{-5})$. We also observe that $P^{(0.8,0.2)}[\tau_{D}>0.0025] = 0.657$ and $P^{(0.8,0.2)}[\tau_{D}>0.005] = 0.34$. Thus, doubling the time reduces the probability of the rumor increasing by more than $10\%$ by almost $32\%$. In fact, in the graph we see that, from a certain point, the rumor grows at an accelerated pace, aligning with intuition.
 
\item[A gonorrhea model:] 

Recall that the techniques developed here can be applied to other stochastic models, as we will demonstrate below.  There are several methods for deriving a stochastic model from a deterministic one, see \cite{mao2007stochastic}. One of the most common and natural methods is the technique known as stochastic parameter variation, which is motivated by the uncertainty in measuring some of the parameters involved in the model.

To present another application of our results, let us consider the SIS model introduced in (see \cite{gray2011stochastic})
\begin{eqnarray*}
\frac{dS}{dt}&=& \mu N- \beta SI+\gamma I- \mu S,\\
\frac{dI}{dt}&=&\beta SI- (\mu + \gamma)I.
\end{eqnarray*}
The parameters $N$, $\mu$, and $\gamma$ are non-negative real numbers, while the parameter $\beta$ is a real number. In this model, $S$ and $I$ represent, respectively, the number of susceptible individuals and the number of infected individuals.

Considering the stochastic variation of the parameter $\beta$ in \cite{gray2011stochastic}, the following stochastic system is obtained
\begin{eqnarray*}
dS&=&[\mu N - \beta SI + \gamma I - \mu S]dt- \alpha SI dW_{1},\\
dI&=&[\beta SI - (\mu + \gamma)I]dt+ \alpha SI dW_{2},
\end{eqnarray*}
from which it follows that the dispersion matrix $(a_{i,j})$ is (see \cite{kuo2006stochastic})
\begin{eqnarray*}
(a_{i,j})=\left(
\begin{tabular}{ c c }
 $\alpha^{2} S^{2}I^{2}$ & $0$ \\ 
 $0$ &  $\alpha^{2} S^{2}I^{2}$
\end{tabular} \right).
\end{eqnarray*}

As an application, let us consider the problem of gonorrhea infection among homosexuals. In \cite{gray2011stochastic}, this infection model is studied with the following parameters:  
$$N=10,000, \quad \mu=6.84463\times 10^{-5}, \quad \gamma=0.018182, \quad \beta=2.55504\times 10^{-6}.$$
Moreover, it is assumed that $S(0)=9,000$ and $I(0)=1,000$. An important problem in such models, produced by stochastic perturbation, is to estimate the perturbed parameters, in this case, the dispersion parameter $\alpha$.

Note that the model is very sensitive to the choice of $\alpha$. To illustrate this, let us consider the set $D=(8,500,9,500)\times (500,1500)$. For example, if $\alpha = 1 \times 10^{-4}$, we observe in Figure \ref{FigGhono}(a) that $E^{(9,000,1,000)}[\tau_{D}]=0.2069$. This result is unrealistic, as it implies that, in less than a day, either the $1,500$ infected individuals were exceeded, or $500$ gonorrhea cases were cured.

Using a logistic model to describe the average behavior of gonorrhea, \cite{de2024certain} proposes taking the value $\alpha=1.5\times 10^{-5}$. In this case, according to Figure \ref{FigGhono}(b), $E^{(9,000,1,000)}[\tau_{D}]=9.091$, which better matches the real data presented in \cite{gray2011stochastic}.

Furthermore, in the probability distributions shown in Figures \ref{FigGhono}(c) and \ref{FigGhono}(d), we see that the first distribution is quite unrealistic, as it indicates that the infection is extremely contagious, leaving the region $(8,500,9,500)\times (500,1,500)$ within minutes. In contrast, in the second distribution, for example, there is a probability of about $40\%$ that it takes the infection more than 8 days to exceed $1,500$ infected individuals or for $500$ to recover, which is more reasonable according to the real model,  \cite{gray2011stochastic}.

As we can see, choosing the correct value of the parameter $\alpha$ is crucial for making accurate predictions. Let us assume that $\alpha = 1.5\times 10^{-5}$. With this value, we know that the infection surpasses $1,500$ infected individuals or more than $500$ recover in an average of nine days. Moreover, from Figure \ref{FigGhono}(d), we see that the probability of the infection leaving the interval $(500,1,500)$ in less than ten days is approximately $67\%$. This is significant because it implies that within ten days, it is quite likely that we will know what action to take.
\begin{figure}[]
    \begin{subfigure}[bt]{0.55\textwidth}
    \centering
       \includegraphics[scale=0.2]{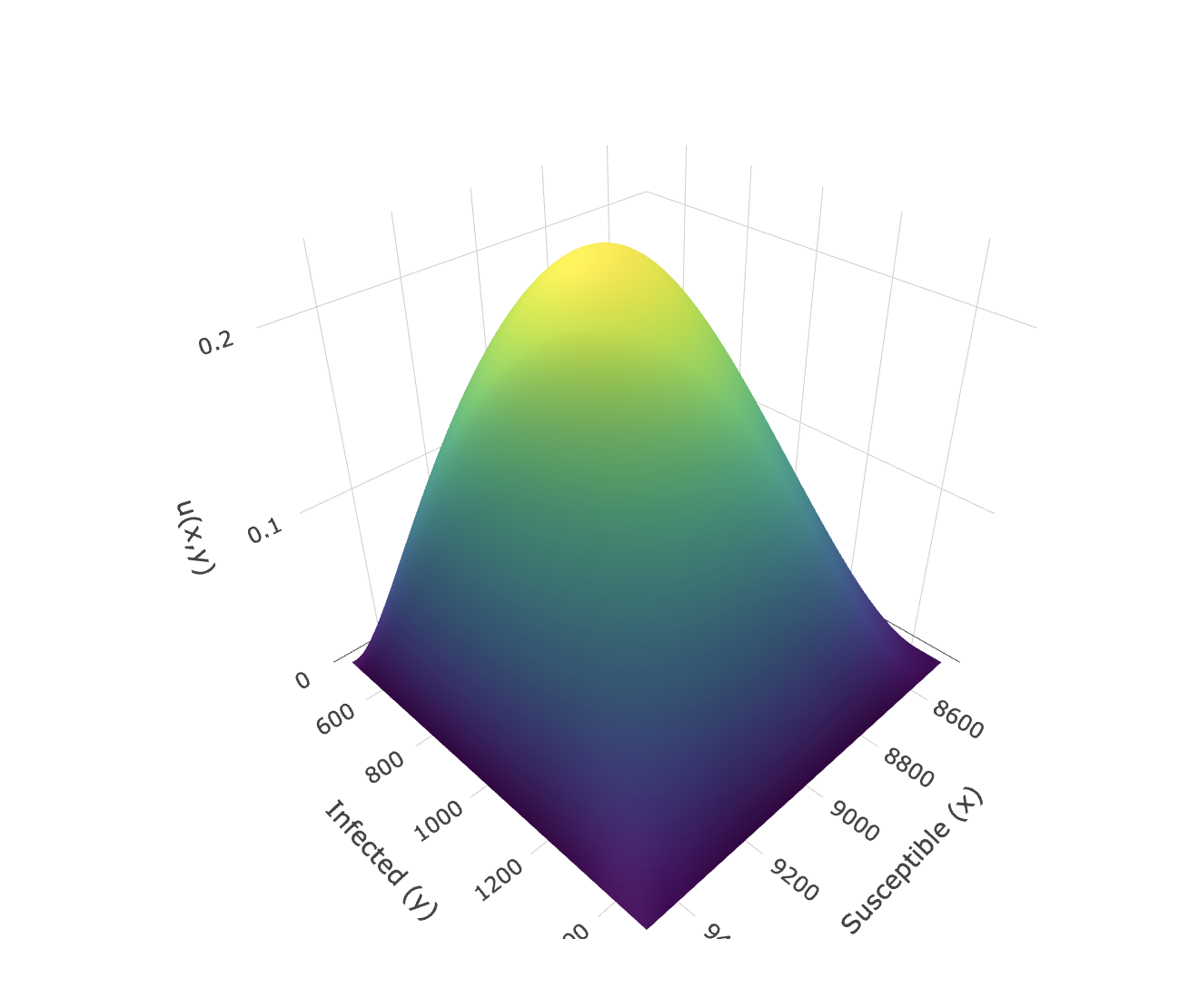}
       \caption{Graph of $x \to E^{x}[\tau_{D}]$, $\alpha=1\times 10^{-4}$.}  
    \end{subfigure}    
    \begin{subfigure}[bt]{0.5\textwidth}
    \hspace{-0.5cm}
       \includegraphics[scale=0.2]{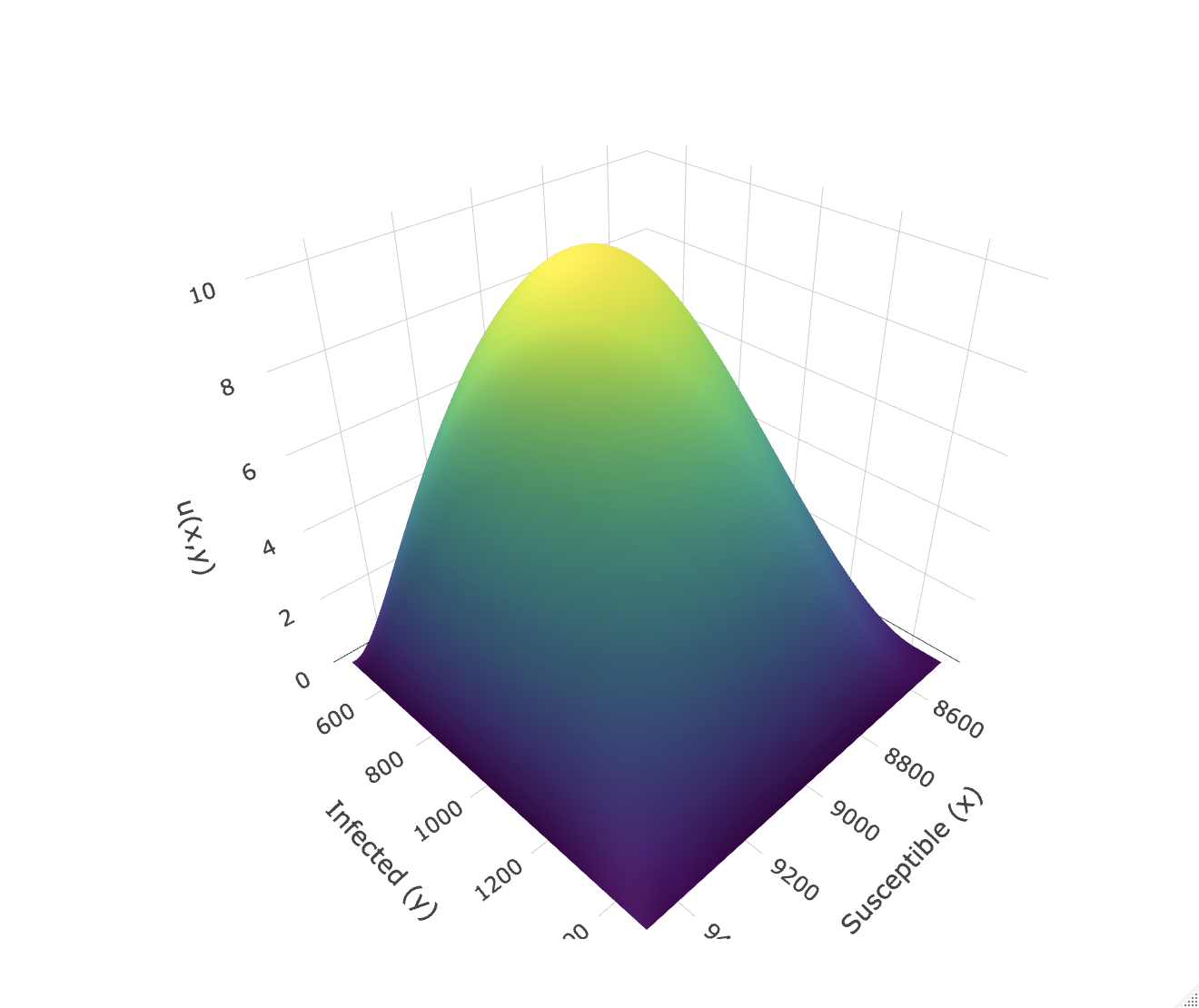}
       \caption{Graph of $x \to E^{x}[\tau_{D}]$, $\alpha=1.5\times 10^{-5}$.}  
    \end{subfigure}

     \begin{subfigure}[bt]{.55\textwidth}
    \hspace{-0.5cm}
       \includegraphics[scale=0.15]{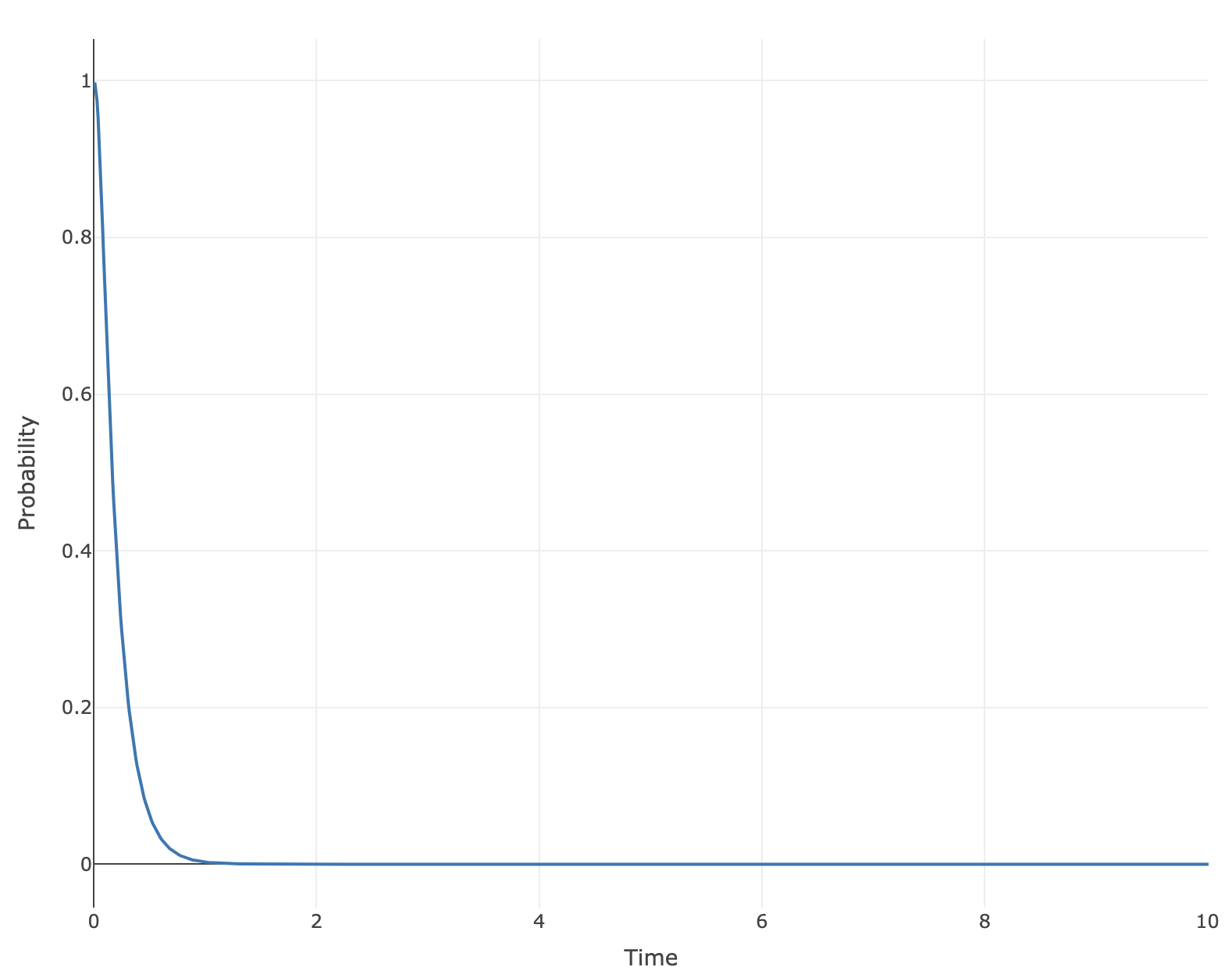}
       \caption{Graph of $t \to P^{(0.8,0.2)}[\tau_{D}>t]$,\\ \centering $\alpha=1\times 10^{-4}$.}  
    \end{subfigure} 
    \begin{subfigure}[bt]{0.4\textwidth}
    \hspace{-1cm}
    \vspace{0.6cm}
       \includegraphics[scale=0.15]{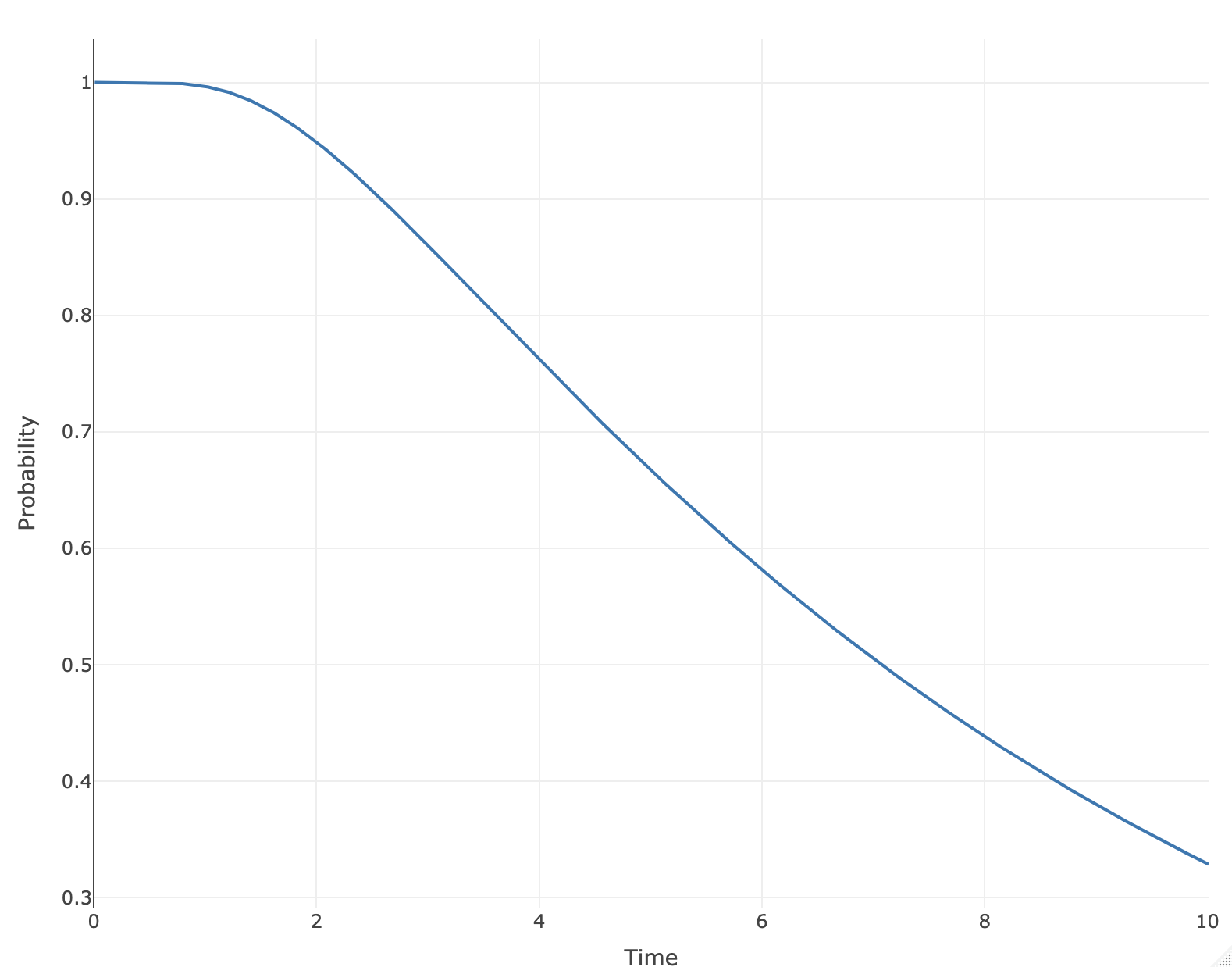}
       \caption{Graph of $t \to P^{(0.8,0.2)}[\tau_{D}>t]$,\\  \centering $\alpha=1.5\times 10^{-5}$.}  
    \end{subfigure}
   \caption{The expectation and probability distribution of the gonorrhea exit time for $D=(8500,9500)\times (500,1500)$.}  \label{FigGhono}  
\end{figure}

\end{description}

Now we will address examples where the domain $D$ is a subset of $\mathbb{R}^{3}$. In this case, it is not possible to graph the functions $(x,y,z) \to u(x,y,z) := E^{(x,y,z)}[\tau_{D}]$ and $(x,y,z) \to v(x,y,z) := P^{(x,y,z)}[\tau_{D} > t_{0}]$. Since we are interested in the expected exit time of a process $Y$, starting at $(x_{0}, y_{0}, z_{0}) \in D$, we will graph the $x_{0}$-section of $u$, that is, the graph of $u_{x_{0}}: D_{x_{0}} \to \mathbb{R}$, where $D_{x_{0}} := {(y,z) : (x_{0},y,z) \in D }$ and $u_{x_{0}}(y,z) := u(x_{0},y,z)$. In this way, we determine $u(x_{0}, y_{0}, z_{0})$. We perform this only for the SIR example, as the method is similar for the model for cancerous tumors. A similar procedure could be applied to $v(x)$. In our case, we will only consider the distribution of the exit time at the point $(x_{0}, y_{0}, z_{0})$, that is, the function $t \to P^{(x_{0},y_{0},z_{0})}[\tau_{D} > t]$.

\begin{description}

\item[The SIR model:] 

Some contagious diseases grant immunity to individuals once they recover. A common model for the spread of such infections is the well-known SIR model. In this case, $S$ represents the proportion of susceptible individuals, $I$ the proportion of infected individuals, and $R$ the proportion of recovered individuals who have acquired immunity to the disease. A classical dynamic for this model is determined by
\begin{eqnarray*}
\frac{dS}{dt}&=&\Lambda -\mu S - \beta SI,\\
\frac{dI}{dt}&=& \beta SI- (\mu+\gamma+\epsilon)I,\\
\frac{dR}{dt}&=&\gamma I - \mu R,
\end{eqnarray*} 
where $\Lambda$, $\mu$, $\beta$, $\gamma$, and $\epsilon$ are non‐negative constants whose specific meanings are explained in detail in \cite{ji2014threshold}; our primary objective here is to determine the exit times of the stochastic model.  

Let $Y=(S,I,R)^{T}$ be the solution of the SIR model. As before, over a small time interval $\Delta t$, the vector $\Delta Y$ takes the following six states
\begin{gather*}
(1,0,0)^{T},\ (-1,0,0)^{T}, \ (-1,1,0)^{T}, \ (0,\mu+\gamma+\epsilon,\gamma)^{T}, \ (0,0,-1), \ (0,0,0)^{T}.
\end{gather*}
Table \ref{TabSIR} shows the transition probabilities for these states.
\begin{table}[H]
\centering
\begin{tabular}{ l | l }
Change & Probability \\ 
 \hline \hline
 $\Delta Y^{1}= (1,0,0)^{T}$ &  $p_{1}:= f_{1}(S) \Delta t$\\
 $\Delta Y^{2}=(-1,0,0)^{T}$ &  $p_{2}:= g_{1}(S) \Delta t$\\
 $\Delta Y^{3}=(-1,1,0)^{T}$ &  $p_{3}:= g_{2}(S,I) \Delta t$\\
 $\Delta Y^{4}=(0,\mu+\gamma+\epsilon,\gamma)^{T}$ &  $p_{4}:= g_{3}(I) \Delta t$\\
  $\Delta Y^{5}=(0,0,-1)^{T}$ &  $p_{5}:= f_{2}(R) \Delta t$\\
 $\Delta Y^{6}=(0,0,0)^{T}$ &  $p_{6}:= 1- \sum_{i=1}^{4} p_{i}$     
\end{tabular}
\caption{Probabilities for changes in the SIR model} \label{TabSIR}
\end{table}
In this case, the functions $f_{i}$ and $g_{i}$ are given by
\begin{gather*} 
f_{1}(S)=\Lambda, \quad g_{1}(S)= \mu S, \quad g_{2}(S,I)= \beta SI,  \quad g_{3}(I)= I, \ f_{2}(R)= \mu R.
\end{gather*}
Thus, the stochastic system modeling the infection is
\begin{equation*}\label{EDESIR}
\left\lbrace 
\begin{tabular}{ l l l}
 $dS(t)$ & $=$ &  $(\Lambda -\mu S - \beta SI) dt +\sum_{i=1}^{3} \sigma_{1,i}(S,I,R) \ dW_{i}(t)$, \\ 
 $dI(t)$ & $=$ &  $(\beta SI- (\mu+\gamma+\epsilon)I) dt + \sum_{i=1}^{3} \sigma_{2,i}(S,I,R) \ dW_{i}(t)$,\\
 $dR(t)$ & $=$ &  $(\gamma I - \mu R) dt + \sum_{i=1}^{3} \sigma_{3,i}(S,I,R) \ dW_{i}(t))$,
\end{tabular} \right.
\end{equation*}
where $(\sigma_{i,j}(S,I,R))$ is the square root of 
\begin{eqnarray*}
(a_{i,j})=\left(
\begin{tabular}{ c c c}
 $\Lambda + \mu S + \beta SI$ & $-\beta SI$ & $0$\\ 
 $-\beta SI$ &  $\beta SI+(\mu+\gamma+\epsilon)^{2}I$ & $(\mu+\gamma+\epsilon)\gamma I$\\
 $0$ &  $(\mu+\gamma+\epsilon)\gamma I$ & $\gamma^{2} I+ \mu R$  
\end{tabular} \right).
\end{eqnarray*}

Next, we will consider the numerical example studied in \cite{ji2014threshold}. The parameters are given by
$$ \Lambda = 5, \quad \mu = 0.95, \quad \beta = 0.8, \quad \gamma = 0.8, \quad \epsilon = 0.6,$$
with initial conditions $(S(0),I(0),R(0))=(0.8,0.1,0.1)$ and domain $D=(0,1)\times (0,1) \times (0,1)$. Using this information, we graph the function $u_{0.8}$, see Figure \ref{FigSIR} (a), and determine that $u(0.8,0.1,0.1)=0.017$. This implies that if the time unit is decades, the entire population recovers in approximately two months. In other words, within about 62 days, all individuals will be infected and eventually acquire immunity. It is important to note that this model does not account for mortality.

Furthermore, from Figure \ref{FigSIR} (b), we obtain the following probabilities:
\begin{center}
\begin{tabular}{| c || c | c | c | c | c | c | c }
 $t$ &  $0.004$ & $0.008$ & $0.013$ & $0.025$ & $0.05$ & $0.1$  \\
 \hline
 $p^{(0.8,0.1,0.1)}[\tau_{D}>t]$ &  $0.8$ & $0.6$ & $0.4$ & $0.2$ & $0.05$ & $0.006$
\end{tabular}
\end{center}
For small times, say $t\leq 0.004$, it is unlikely that the infection will leave the region $D$, meaning the infectious process is still ongoing. Even if the time is doubled to $t\leq 0.008$, there is still a $40\%$ probability that the infection continues affecting the population. However, as time progresses, the probability $P^{(0.8,0.1,0.1)}[\tau_{D}>t]$ decreases rapidly. Indeed, for times greater than $0.013$, the probability that the disease continues affecting the population drops to $60\%$, reflecting the highly contagious nature of the infection. This is evident from the steep curve in the distribution of $\tau_{D}$ in Figure \ref{FigSIR} (b).

Additionally, from Figure \ref{FigSIR} (a), we observe that $u(0.8,0.04,0.16)=0.01$, which implies that a $40\%$ reduction in the disease increases the average exit time by $59\%$. The convex shape of the graph of $u_{0.8}$ indicates that swift intervention with health measures can significantly alter the exit time from the region $D$. Furthermore, it is very unlikely (less than $0.006$) that the infection persists for more than $0.1$ time units. The rapid progression of the infection suggests that implementing herd immunity would be highly advisable. After a brief chaotic period, the population could resume activities to compensate for the losses incurred during the infection stage. Once again, we emphasize that the model does not consider potential fatalities. This model could be applied, for example, to a concentrated population such as a school, dealing with a non-lethal infectious disease like the flu.   
\begin{figure}[]
    \begin{subfigure}[bt]{0.55\textwidth}
    \centering
    \hspace{-2cm}
       \includegraphics[scale=0.19]{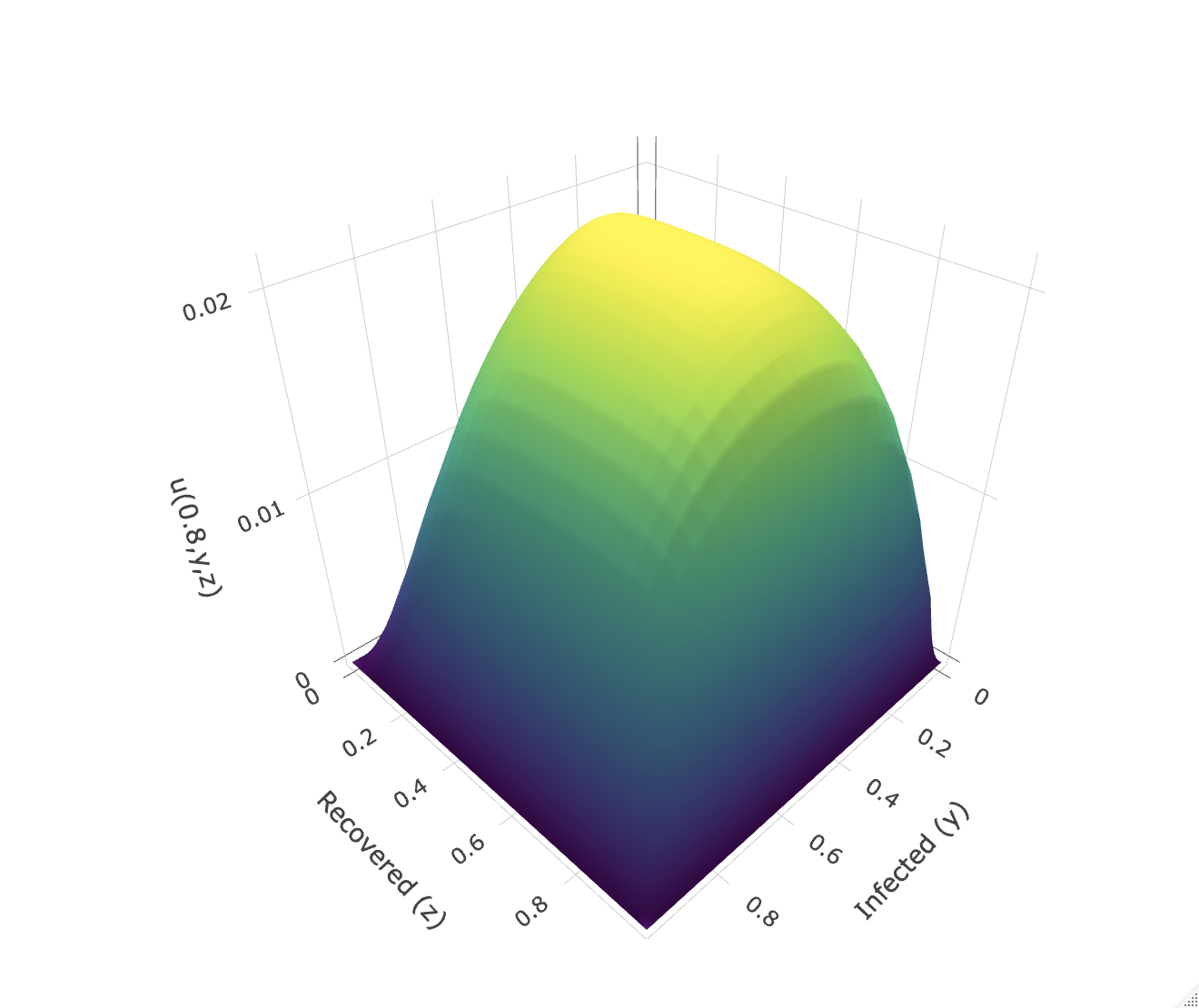}
       \caption{Graph of $(0.8,y,z) \to E^{(0.8,y,z)}[\tau_{D}]$.}  
    \end{subfigure}    
    \begin{subfigure}[bt]{0.5\textwidth}
    \centering
    \vspace{1cm}
    \hspace{-1.9cm}
       \includegraphics[scale=0.17]{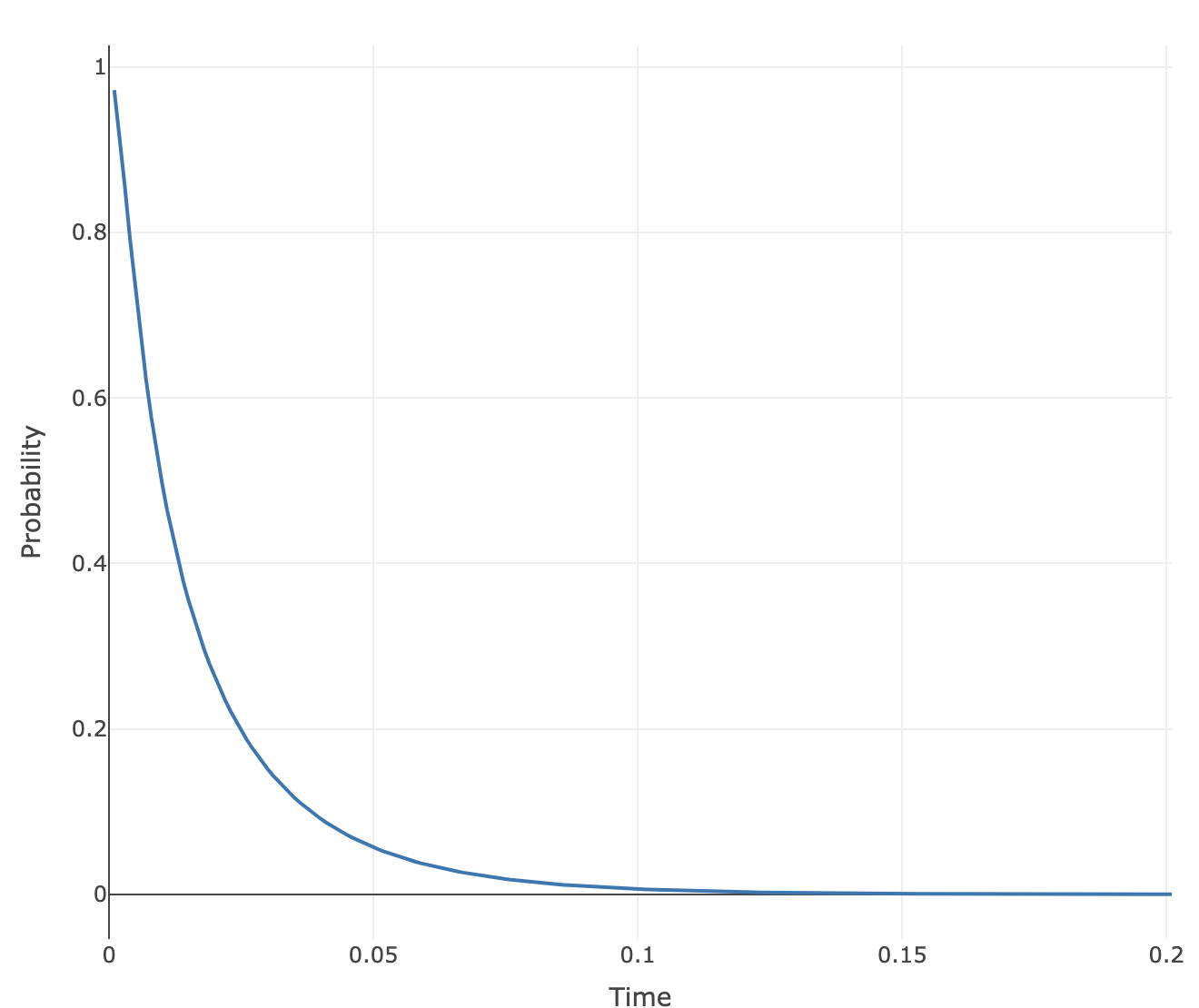}
      \hspace{-1cm} 
       \caption{Graph of $t \to P^{(0.8,0.2,0.6)}[\tau_{D}>t]$.}  
    \end{subfigure}
   \caption{The expectation and probability distribution of the SIR exit time for $D=(0,1)\times (0,1) \times (0,1)$.}  \label{FigSIR}  
\end{figure}

\item[A model for cancerous tumors.] 

As we have noted, to obtain the stochastic dynamics associated with a deterministic model, the dispersion matrix $\sigma$ is necessary. This matrix is the square root of the positive definite matrix $A$, see equation (\ref{matA}). The existence of $\sigma$ is generally established; however, except for the case $n=2$, it is not easy to explicitly determine its value, see \cite{allen2007modeling}. Some authors, aiming to obtain a stochastic model, consider an auxiliary matrix $\tilde{\sigma}$ such that $A=\tilde{\sigma}\tilde{\sigma}^{T}$, see \cite{allen2008construction}. This approach induces greater volatility in the system, as it requires more than $n$ Brownian motions. Consequently, the resulting stochastic system may be of higher dimension, making the application of numerical methods to solve systems (\ref{PDir}) and (\ref{PVI}) more challenging. However, as we have previously noted, our results do not require explicit knowledge of the matrix $\sigma$ but only of the matrix $A$, which can always be explicitly determined.

To illustrate this point, we have considered the model of cancerous tumor growth, which is described in detail in \cite{liu2018deterministic}. The deterministic dynamics of the model are given by the following system of differential equations
\begin{eqnarray*}
\frac{dE}{dt}&=& s+\frac{\rho ET}{\alpha+T}-\beta_{1} ET-d_{1}E-c_{1}E,\\
\frac{dN}{dt}&=& r_{1}N(1-b_{1}N)-\beta_{2}NT-c_{2}N,\\
\frac{dT}{dt}&=& r_{2}T(1-b_{2}T)-\beta_{3}ET-\beta_{4}NT-c_{3}T,
\end{eqnarray*}
where the parameters $s$, $\rho$, $\alpha$, $\beta_{1}$, $d_{1}$, $c_{1}$, $r_{1}$, $b_{1}$, $\beta_{2}$, $c_{2}$, $r_{2}$, $b_{2}$, $\beta_{3}$, $\beta_{4}$, and $c_{3}$ are non-negative real numbers. The variable $E$ denotes the population of effector immune cells, $N$ represents the population of normal cells, and $T$ is the population of tumor cells.

Let us denote the tumor growth dynamics by $Y=(E,N,T)^{T}$. Following the scheme of the previous examples, we know that over a small time interval $\Delta t$, the vector $\Delta Y$ can take only the following five states
\begin{gather*}
(1,0,0)^{T}, \ (-\beta_{1},0,-\beta_{3})^{T}, \ (-1,0,0)^{T}, \ (0,1,0)^{T}, \\
(0,-\beta_{2},-\beta_{4})^{T}, \ (0,-1,0), \ (0,0,1)^{T}, \ (0,0,-1)^{T}, \ (0,0,0)^{T}.
\end{gather*}
The probability assignments to the corresponding changes are shown in Table \ref{TabCan}.
\begin{table}[]
\centering
\begin{tabular}{ l | l }
Change & Probability \\ 
 \hline \hline
 $\Delta Y^{1}=(1,0,0)^{T}$ &  $p_{1}:= f_{1}(E,T) \Delta t$\\
 $\Delta Y^{2}=(-\beta_{1},0,-\beta_{3})^{T}$ &  $p_{2}:= g_{1}(E,T) \Delta t$\\
 $\Delta Y^{3}=(-1,0,0)^{T}$ &  $p_{3}:= g_{2}(E) \Delta t$\\
 $\Delta Y^{4}=(0,1,0)^{T}$ &  $p_{4}:= f_{2}(N) \Delta t$\\
 $\Delta Y^{5}=(0,-\beta_{2},-\beta_{4})^{T}$ &  $p_{5}:= g_{3}(N,T) \Delta t$\\
 $\Delta Y^{6}=(0,-1,0)^{T}$ &  $p_{6}:= g_{4}(N) \Delta t$\\
 $\Delta Y^{7}=(0,0,1)^{T}$ &  $p_{7}:= f_{3}(T) \Delta t$\\
 $\Delta Y^{8}=(0,0,-1)^{T}$ &  $p_{8}:= g_{5}(T) \Delta t$\\
 $\Delta Y^{9}=(0,0,0)^{T}$ &  $p_{9}:= 1- \sum_{i=1}^{8} p_{i}$     
\end{tabular}
\caption{Probabilities for changes in the cencerous tumor model.} \label{TabCan}
\end{table}

On the other hand, the functions $f_{i}$ and $g_{i}$ are given as follows
\begin{gather*} 
f_{1}(E,T)=s+\frac{\rho ET}{\alpha+T}, \quad g_{1}(E,T)= ET, \quad g_{2}(E)= (d_{1}+c_{1})E,  \quad f_{2}(N)= r_{1}N(1-b_{1}N), \\
g_{3}(N,T)= NT, \quad g_{4}(N)=c_{2}N, \quad f_{3}(T)= r_{2}T(1-b_{2}T), \quad g_{5}(T)= c_{3}T.
\end{gather*}
Thus, our stochastic system of interest is expressed as follows 
\begin{equation}\label{EDETumor}
\left\lbrace 
\begin{tabular}{ l l l}
 $dE(t)$ & $=$ &  $(s+\frac{\rho ET}{\alpha+T}-\beta_{1} ET-d_{1}E-c_{1}E) dt + \sum_{i=1}^{3} \sigma_{1,i}(E,N,T) \ dW_{i}(t)$, \\ 
 $dN(t)$ & $=$ &  $(r_{1}N(1-b_{1}N)-\beta_{2}NT-c_{2}N) dt + \sum_{i=1}^{3} \sigma_{2,i}(E,N,T) \ dW_{i}(t)$,\\
 $dT(t)$ & $=$ &  $(r_{2}T(1-b_{2}T)-\beta_{3}ET-\beta_{4}NT-c_{3}T) dt + \sum_{i=1}^{3} \sigma_{3,i}(E,N,T) \ dW_{i}(t)$,
\end{tabular} \right.
\end{equation}
where $\sigma=(\sigma_{i,j}(E,N,T))$ is the square root of the matrix $A=(a_{i,j})$, where
\begin{align*}
&a_{1,1}= s+\frac{\rho ET}{\alpha+T}+\beta_{1}^{2}ET+(d_{1}+c_{1})E, \\
&a_{1,2}= a_{2,1}= 0, \\
&a_{1,3}= a_{3,1}= \beta_{1} \beta_{3}ET, \\
&a_{2,2}= r_{1}N(1-b_{1}N)+\beta_{2}^{2}NT+c_{2}N, \\
&a_{2,3}= a_{3,2}= \beta_{2} \beta_{4} NT, \\
&a_{3,3}= \beta_{3}^{2}ET+\beta_{4}^{2}NT+c_{3}T+ r_{2}T(1-b_{2}T).
\end{align*}

The system (\ref{EDETumor}) requires only three Brownian motions, unlike the stochastic system studied in \cite{liu2018deterministic}, which requires eight Brownian motions. We do not know of a generic software to solve the Partial Differential Equations corresponding to types (\ref{PDir}) and (\ref{PVI}). However, if necessary, an explicit numerical method would need to be applied to solve them. Therefore, it is advisable not to introduce more volatility into the deterministic system than necessary, in order to apply well known numerical software.

To provide a numerical example, let us consider the case studied in \cite{liu2018deterministic} with the following parameters
\begin{gather*}
s=1, \quad \rho=0.3, \quad \alpha= 0.8, \quad \beta_{1}= 1, \quad d_{1}=0.3, \quad c_{1}=0.2, \quad r_{1}=0.7, \\
b_{1}=0.6, \quad \beta_{2}=0.1, \quad c_{2}=0.2, \quad r_{2}=2.3 \quad b_{2}=0.2, \quad \beta_{3}=0.3 \quad \beta_{4}=0.3, \quad c_{3}=0.2.
\end{gather*}
The initial value is the point $(E(0),N(0),T(0))=(3,1.5,1)$, and the domains are $D_{1}=(0,4)\times (0,2) \times (0,2)$ and $D_{2}=(0,4)\times (0,2) \times (0,4)$. From Figure \ref{FigTumor} (a) and (c), we observe, respectively, that $E^{(3,1.5,1)}[\tau_{D_{1}}]=0.38$ and $E^{(3,1.5,1)}[\tau_{D_{2}}]=0.6$. The modification of the domains consisted of allowing the number of tumor cells to double without varying the ranges of the other variables. With this change, the exit time from domain $D_{2}$ increased by $79\%$, and this increase of $0.22$ time units was exclusively due to now permitting a greater number of cancer cells. Assuming that the other variables, $E$ and $N$, lie in $(0,4)$ and $(0,2)$, respectively, then variable $T$ takes $0.38$ time units to exit the interval $(0,2)$, while exiting the interval $(0,4)$ requires an additional $0.22$ time units. This moderate increase in the average exit time gives us an indication that the growth of cancer cells is relatively slow. Moreover, from Figure \ref{FigTumor} (b) and (d), we see that:
\begin{center}
\begin{tabular}{| c || c | c | c | c | c | c | c | c | c | }
 $t$ &  $0.1$ & $0.2$ & $0.3$ & $0.4$ & $0.5$ & $0.6$ \\
 \hline
 $P^{(0.8,0.1,0.1)}[\tau_{D_{1}}>t]$ &  $0.58$ & $0.35$ & $0.23$ & $0.15$ & $0.11$ & $0.07$\\ \hline
 $P^{(0.8,0.1,0.1)}[\tau_{D_{2}}>t]$ &  $0.63$ & $0.45$ & $0.35$ & $0.27$ & $0.21$ & $0.17$\\
\end{tabular}
\end{center}
which indicates that at the beginning, the exit times are relatively close, that is, the process $Y=(E,N,T)$ exits via the first two variables, and as time progresses, the probability of the exit time from $D_{2}$ is twice that of $D_{1}$, reflecting the increase in the range of $T$. From here, we can deduce that for this cancer tumor model, the growth of cancer cells is severe. Indeed, the small probabilities of $P^{(0.8,0.1,0.1)}[\tau_{D_{i}}>t]$, $t\geq 0.5$, indicate that the process quickly leaves the domains $D_{i}$, $i=1,2$. This seems to indicate that early-stage anticancer treatment is most advisable.
\begin{figure}[]
    \begin{subfigure}[bt]{0.5\textwidth}
    \centering
       \includegraphics[scale=0.19]{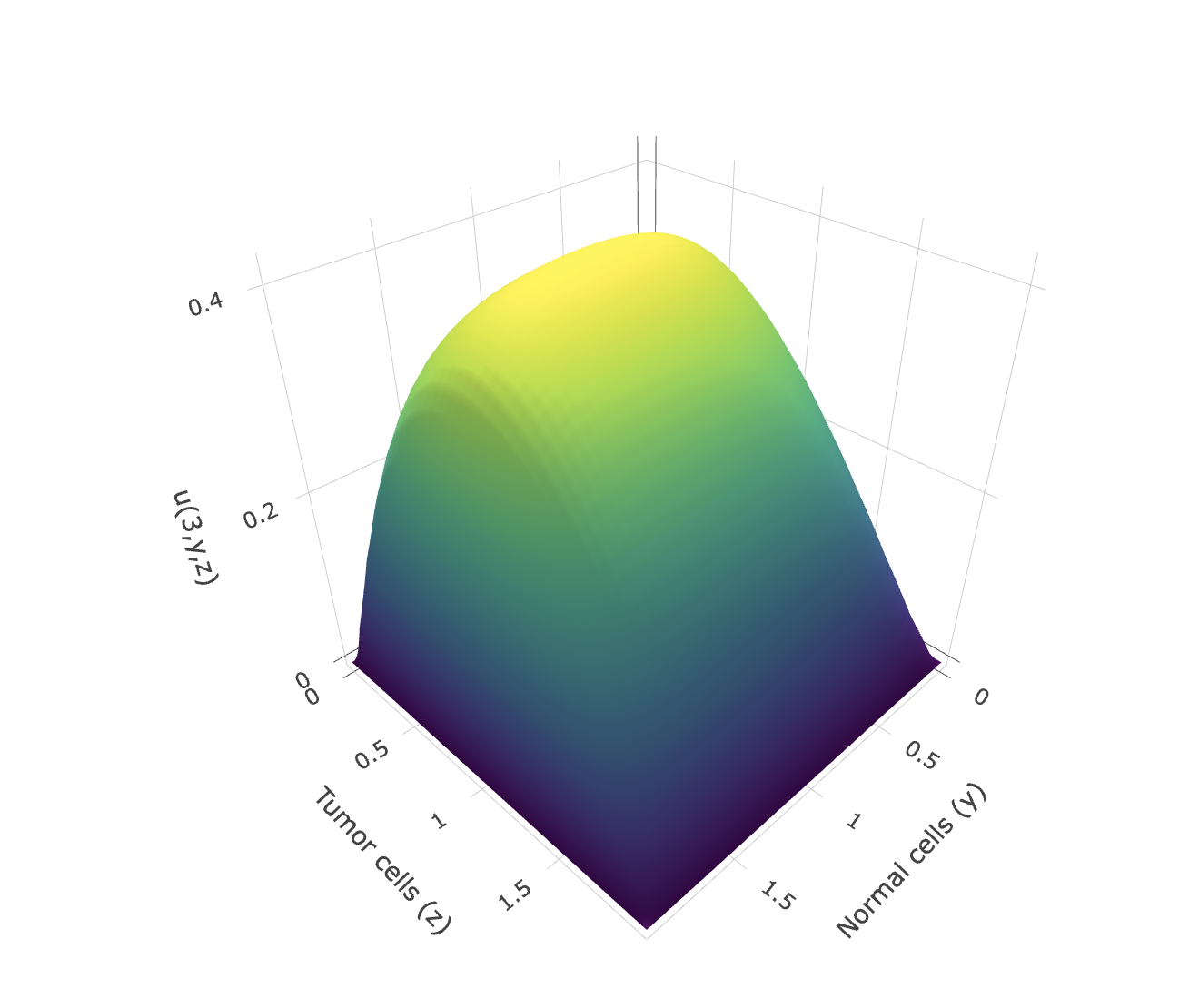}
       \caption{Graph of $(3,y,z) \to E^{(3,y,z)}[\tau_{D_{1}}]$.}  
   \end{subfigure}    
    \begin{subfigure}[bt]{0.5\textwidth}
    \centering
     \vspace{1cm}
       \includegraphics[scale=0.16]{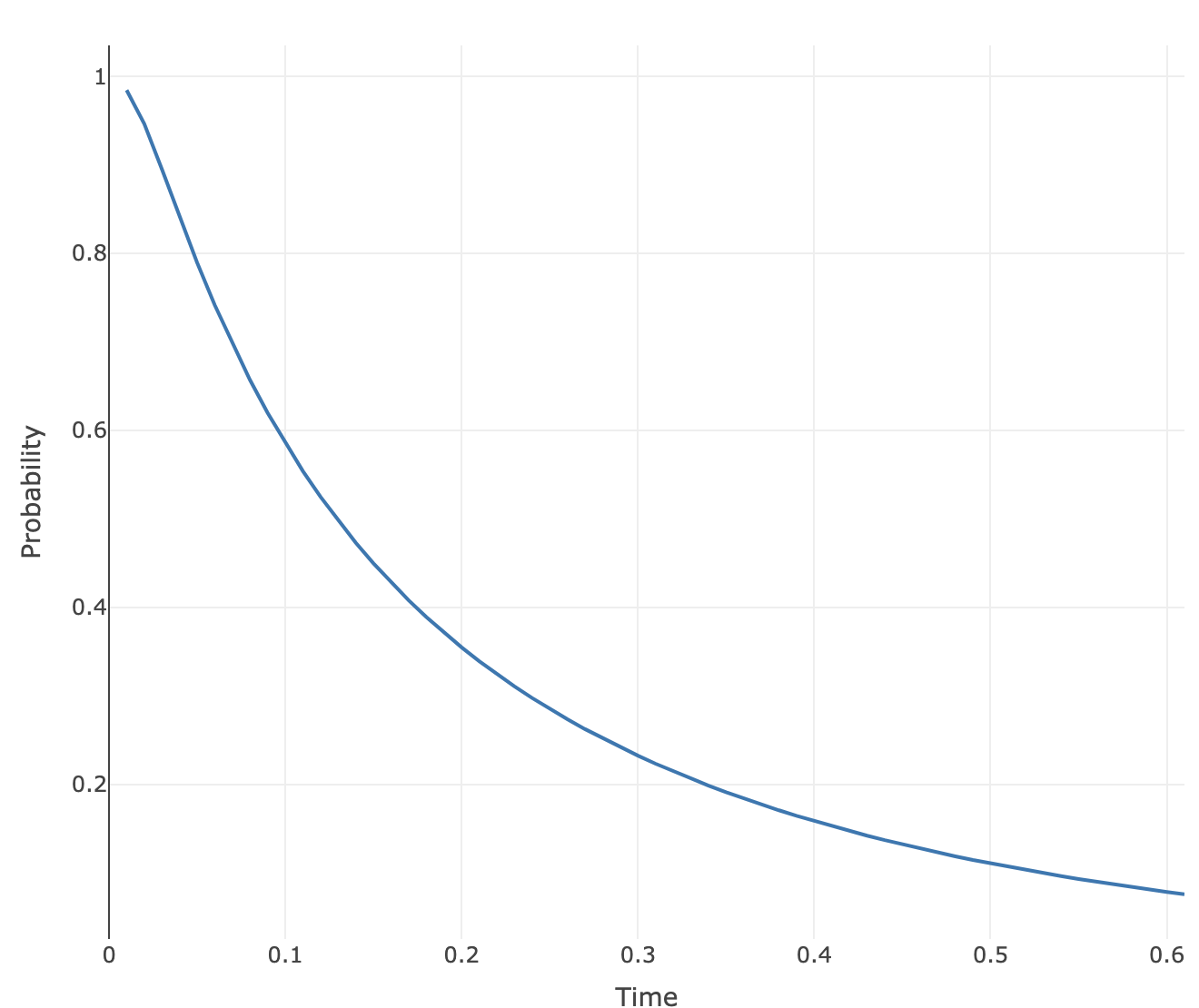}
       \caption{Graph of $t \to P^{(3,1.5,1)}[\tau_{D_{1}}>t]$.}  
    \end{subfigure}
   \begin{subfigure}[bt]{0.5\textwidth}
    \centering
       \includegraphics[scale=0.19]{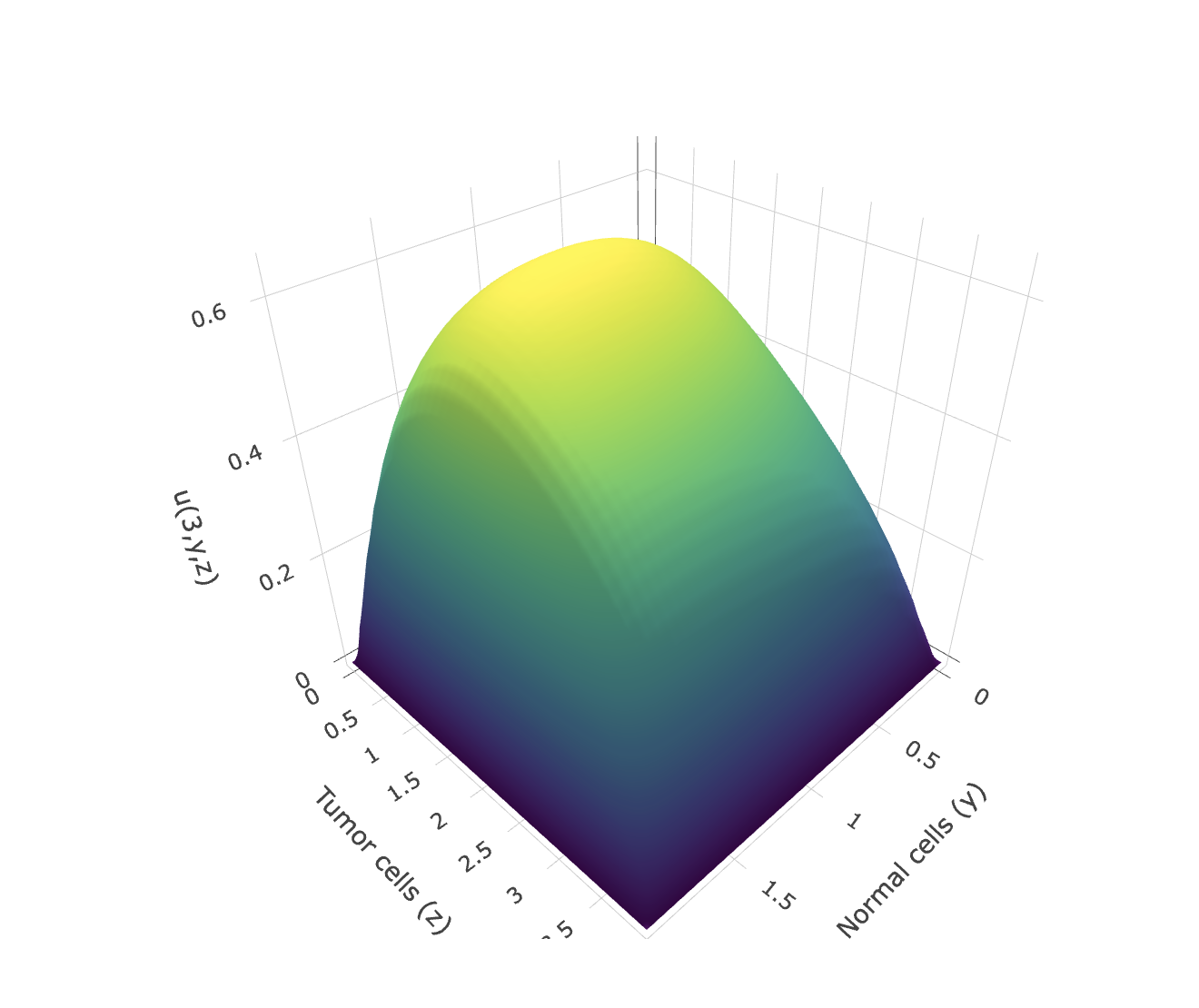}
       \caption{Graph of $(3,y,z) \to E^{(3,y,z)}[\tau_{D_{2}}]$.}  
   \end{subfigure}    
    \begin{subfigure}[bt]{0.5\textwidth}
    \centering
     \vspace{1cm}
       \includegraphics[scale=0.16]{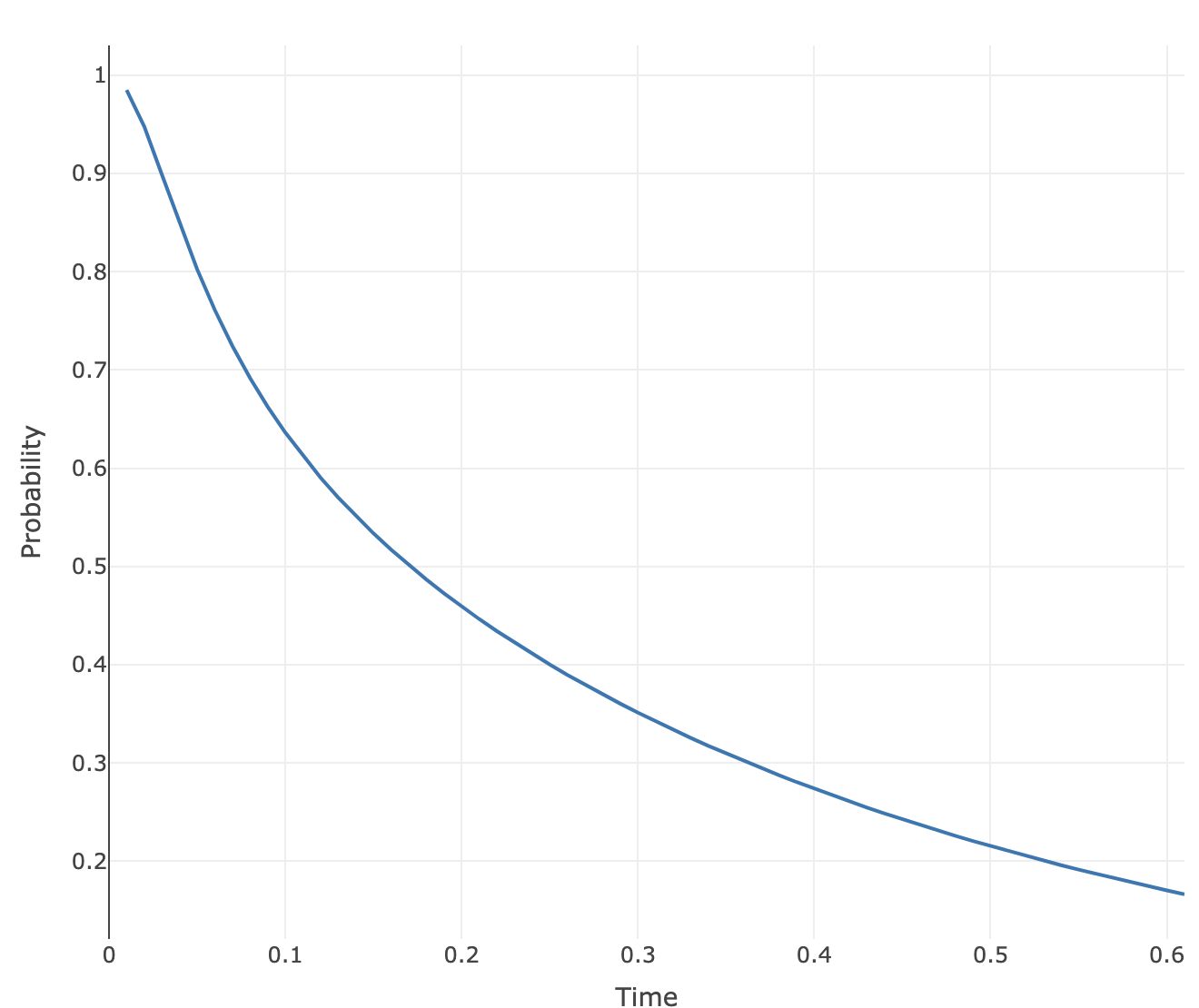}
       \caption{Graph of $t \to P^{(3,1.5,1)}[\tau_{D_{2}}>t]$.}  
    \end{subfigure} 
   \caption{The expectation and probability distribution for the exit time  of $D_{1}$ and $D_{2}$ for the tumor model.}  \label{FigTumor}  
\end{figure}
\end{description}

\section{Conclusions and Perspectives}

In this work, we have used the “perturbation by the states of the dynamical system” technique to extend certain deterministic models into stochastic frameworks, thereby accounting for random perturbations.  This is just one of many methods for generating stochastic models: starting from a system of ordinary differential equations, we assume that, over sufficiently small time intervals, the system transitions between states with specified probabilities (see \cite{allen2007modeling}).  Employing stochastic‐calculus techniques, we then derive two fundamental quantities: the mean exit time and the exit‐time distribution of the stochastic process from a bounded domain.  These quantities are obtained by solving elliptic and parabolic partial differential equations, respectively.  

We demonstrated the feasibility of constructing a stochastic model from a deterministic one, ensuring that the average behavior of the stochastic system faithfully reproduces that of its deterministic counterpart. Furthermore, we proposed a numerical scheme implemented in FreeFEM, suitable for 2D and 3D problems, which bridges theoretical results with practical applications. The illustrative examples provided highlight the effectiveness of our approach.

While our current focus is limited to 2D and 3D settings due to software constraints, future research could explore numerical methods capable of addressing problems in higher dimensions. Advancing specialized software for solving partial differential equations in higher-dimensional settings would significantly expand the applicability and impact of this work. Additionally, integrating artificial intelligence techniques to optimize and expedite the solution of partial differential equations offers a promising avenue for future research.

In summary, this study provides a robust foundation for analyzing stochastic processes in multidimensional domains, offering both theoretical insights and practical tools for their application. Future directions include extending the framework to higher dimensions, developing advanced software solutions, and leveraging AI-driven methodologies to enhance computational efficiency.

\section{Appendix}
 
Below, we present the FreeFem codes used in the example of cancer cell growth. The codes for the other examples can be adapted using this as a reference.
 
\bigskip
 
\noindent // Code to solve the equation (\ref{PDir}) \\

\noindent // Definition of the domain D=(x1,x2) $\times$ (y1,y2) $\times$ (z1,z2) and its mesh \\
\noindent load "msh3"  // Load the module that allows you to work with three-dimensional meshes\\
real x1 = 0, x2 = 4, y1 = 0, y2 = 2, z1 = 0, z2 = 4; \\
int k = 40; // Number of divisions of the sides of D \\
mesh3 Th = cube(k, k, k, [x*(x2-x1)+x1,y*(y2-y1)+y1,z*(z2-z1)+z1]); \\

\noindent // Parameter definitions \\
real ss = 1, ro = 0.3, alfa = 0.8, beta1 = 1, d1 =0.3, c1 = 0.2, r1 = 0.7, bb1 = 0.6; \\
real beta2 = 0.1, c2 = 0.2, r2 = 2.3, bb2 = 0.2, beta3 = 0.3, beta4 = 0.3, c3 = 0.2; \\

\noindent // Definition of functions  \\
func a11  = ss + ro*x*z/(alfa+z)+ beta1*beta1*x*z + (d1+c1)*x;  \\
func a12  = 0; \\
func a13  = beta1*beta3*x*z; \\
func a21  = 0; \\
func a22  = r1*y*(1-bb1*y)+beta2*beta2*y*z + c2*y; \\
func a23  = beta2*beta4*y*z;  \\
func a31  = beta1*beta3*x*z;  \\
func a32  = beta2*beta4*y*z;  \\
func a33  = beta3*beta3*x*z + beta4*beta4*y*z + c3*z + r2*z*(1-bb2*z);  \\
func da11  = ro*z/(alfa+z)+beta1*beta1*z+d1+c1; \\
func da12  = 0; \\
func da13  = beta1*beta3*x; \\
func da21  = 0; \\
func da22  = r1-2*r1*bb1*y+beta2*beta2*z+c2; \\
func da23  = beta2*beta4*y; \\
func da31  = beta1*beta3*z; \\
func da32  = beta2*beta4*z;  \\
func da33  = beta3*beta3*x+beta4*beta4*y+c3+r2-2*r2*bb2*z;  \\
func b1 = ss+ro*x*z/(alfa+z)-beta1*x*z-(d1+c1)*x; \\
func b2 = r1*y*(1-bb1*y)-beta2*y*z-c2*y; \\
func b3 = r2*z*(1-bb2*z)-beta3*x*z-beta4*y*z-c3*z; \\

\noindent fespace Vh(Th, P1); // Defines the function space on the three-dimensional mesh \\
Vh u, w; \\
func g = 0;  // Define the boundary condition \\

\noindent // Define the variational problem \\
problem Problem(u, w) = int3d(Th)( \\
       0.5*a11*dx(u)*dx(w) + 0.5*a12*dx(u)*dy(w) + 0.5*a13*dx(u)*dz(w) \\
     + 0.5*a21*dy(u)*dx(w) + 0.5*a22*dy(u)*dy(w) + 0.5*a23*dy(u)*dz(w) \\
     + 0.5*a31*dz(u)*dx(w) + 0.5*a32*dz(u)*dy(w) + 0.5*a33*dz(u)*dz(w) \\
     + 0.5*da11*dx(u)*w + 0.5*da12*dx(u)*w + 0.5*da13*dx(u)*w  \\
     + 0.5*da21*dy(u)*w + 0.5*da22*dy(u)*w + 0.5*da23*dy(u)*w  \\
     + 0.5*da31*dz(u)*w + 0.5*da32*dz(u)*w + 0.5*da33*dz(u)*w  \\
     - b1*dx(u)*w - b2*dy(u)*w - b3*dz(u)*w) // Bilinear part  \\
  - int3d(Th)(w) // Linear part  \\
  + on(1, 2, 3, 4, 5, 6, u = g); // Dirichlet conditions on all edges \\

\noindent Problem; // Solve the problem \\
\noindent ofstream fout("TumorE.txt"); // Open the file for writing \\
\noindent real x0 = 3; // Define the section to graph \\
\noindent real epsilon = 1e-6;  // Defines the tolerance for the comparison \\

\noindent // Evaluate u on the plane $x \approx  x0$ \\
for (int i = 0; i $<$ Th.nv; i++) \{ \\
    real xx = Th(i).x; // x coordinate of the vertex \\
    real yy = Th(i).y; // y coordinate of the vertex \\
    real zz = Th(i).z; // z coordinate of the vertex \\
    // Just save the points where $x \approx x0$ \\
    if (abs(xx - x0) $<$ epsilon) \{   \\     
        real uVal = u(Th(i).x, Th(i).y, Th(i).z); // Evaluate u at the vertex \\
        fout $<<$ yy $<<$ " " $<<$ zz $<<$ " " $<<$ uVal $<<$ endl; // Save the vector (y, z, u) \\
    \} \ 
\}

\bigskip

We will now provide the FreeFEM code to solve the Parabolic Differential Equation.

\bigskip

\noindent // Code to solve the equation (\ref{PVI}) \\

\noindent // Definition of the domain D=(x1,x2) $\times$ (y1,y2) $\times$ (z1,z2) and its mesh \\
\noindent load "msh3" \\
\noindent real x1 = 0, x2 = 4, y1 = 0, y2 = 2, z1 = 0, z2 = 4; \\
int k = 40; // Number of divisions of the sides of D \\
mesh3 Th = cube(k, k, k, [x*(x2-x1)+x1,y*(y2-y1)+y1,z*(z2-z1)+z1]); \\

\noindent // Definition of parameters \\
real ss = 1, ro = 0.3, alfa = 0.8, beta1 = 1, d1 =0.3, c1 = 0.2, r1 = 0.7, bb1 = 0.6; \\
real beta2 = 0.1, c2 = 0.2, r2 = 2.3, bb2 = 0.2, beta3 = 0.3, beta4 = 0.3, c3 = 0.2; \\

\noindent // Definition of functions \\
func a11  = ss + ro*x*z/(alfa+z)+ beta1*beta1*x*z + (d1+c1)*x; \\
func a12  = 0; \\
func a13  = beta1*beta3*x*z; \\
func a21  = 0; \\
func a22  = r1*y*(1-bb1*y)+beta2*beta2*y*z + c2*y; \\
func a23  = beta2*beta4*y*z; \\
func a31  = beta1*beta3*x*z; \\
func a32  = beta2*beta4*y*z; \\
func a33  = beta3*beta3*x*z + beta4*beta4*y*z + c3*z + r2*z*(1-bb2*z); \\
func da11  = ro*z/(alfa+z)+beta1*beta1*z+d1+c1; \\
func da12  = 0; \\
func da13  = beta1*beta3*x; \\
func da21  = 0; \\
func da22  = r1-2*r1*bb1*y+beta2*beta2*z+c2; \\
func da23  = beta2*beta4*y; \\
func da31  = beta1*beta3*z; \\
func da32  = beta2*beta4*z; \\
func da33  = beta3*beta3*x+beta4*beta4*y+c3+r2-2*r2*bb2*z; \\
func b1 = ss+ro*x*z/(alfa+z)-beta1*x*z-(d1+c1)*x; \\
func b2 = r1*y*(1-bb1*y)-beta2*y*z-c2*y; \\
func b3 = r2*z*(1-bb2*z)-beta3*x*z-beta4*y*z-c3*z; \\

\noindent func g = 0; // Definition of boundary condition \\
real dt = 0.01; // Step size over time \\

\noindent // Definition of the test function and the solution \\
fespace Vh(Th, P1); // Defines the space of interpolation functions \\
Vh u, w, uold; // Variables for the solution \\

\noindent // Define the variational problem \\
problem Proba(u, w) = int3d(Th)( \\
       dt*0.5*a11*dx(u)*dx(w) + dt*0.5*a12*dx(u)*dy(w) + dt*0.5*a13*dx(u)*dz(w) \\
     + dt*0.5*a21*dy(u)*dx(w) + dt*0.5*a22*dy(u)*dy(w) + dt*0.5*a23*dy(u)*dz(w) \\
     + dt*0.5*a31*dz(u)*dx(w) + dt*0.5*a32*dz(u)*dy(w) + dt*0.5*a33*dz(u)*dz(w) \\
     + dt*0.5*da11*dx(u)*w + dt*0.5*da12*dx(u)*w + dt*0.5*da13*dx(u)*w \\
     + dt*0.5*da21*dy(u)*w + dt*0.5*da22*dy(u)*w + dt*0.5*da23*dy(u)*w \\
     + dt*0.5*da31*dz(u)*w + dt*0.5*da32*dz(u)*w + dt*0.5*da33*dz(u)*w \\
     - dt*b1*dx(u)*w - dt*b2*dy(u)*w - dt*b3*dz(u)*w +w*u) // Bilinear part \\
  - int3d(Th)(w*uold) // Linear part \\
  + on(1, 2, 3, 4, 5, 6, u = g); // Dirichlet conditions on all edges \\

\noindent // Initial conditions \\
u = 1; // Initial temperature throughout the domain \\
uold = 0; // Initialize uold  \\

\noindent // Find the index of the node closest to (xini, yini, zini) \\
real xini = 0.3, yini = 1.5, zini = 1;\\
int closestNode = -10;\\
real minDist = 1e30;\\
\noindent for (int j = 0; j $<$ Th.nv; j++) \{\\
    real dist = sqrt((Th(j).x - xini)$\wedge$2 + (Th(j).y - yini)$\wedge$2 + (Th(j).z - zini)$\wedge$2);\\
    if (dist $<$ minDist) \{ \\
        minDist = dist; \\
        closestNode = j; \\
    \} \
\}

\noindent // File to save the temporal evolution in (xini, yini, zini) \\
ofstream fout("TumorP.txt");\\

\noindent // Time loop \\
real T = 0.6; // Time interval [0, T] \\
real t = 0;   \\ 
for (int i = 0; i $<=$ (T/dt); i++) \{  \\
  t = t + dt;  \\
  uold = u; // Update the previous solution  \\
  Proba; // Solve the problem \\
// Store the value of u at the node closest to (xini, yini, zini) at this time  \\
 fout $<<$ t $<<$ " " $<<$ u[ ][closestNode] $<<$ endl;
\}

\section*{Acknowledgments}

The author J.V.-M. received partial support from grant PIM25-2 from the Universidad Aut\'onoma de Aguascalientes and Secihti.

\bibliographystyle{unsrt}

\bibliography{PrModels_III}

\end{document}